\documentclass[preprint,12pt]{elsarticle}

\usepackage{amssymb}
\usepackage{amsthm}
\usepackage{xcolor}
\usepackage{amsmath}
\usepackage{color}
\usepackage{multirow}
\usepackage{booktabs}

\usepackage{amsmath}
\usepackage{amsfonts}
\usepackage{amssymb}
\usepackage{color}
\usepackage{dsfont}
\usepackage{float}
\usepackage{multirow}
\usepackage{booktabs}
\usepackage{bigstrut}

\newtheorem{thm}{Theorem}[section]
\newtheorem{lem}{Lemma}[section]
\theoremstyle{definition}
\newtheorem{rem}{\noindent Remark}[section]

\begin{document}
\begin{frontmatter}

\title{\LARGE Corrected Discrete Approximations for the \vspace*{1mm}\\Conditional and Unconditional Distributions of the \vspace*{1mm}\\Continuous Scan Statistic\vspace*{5mm}}

\author{Yi-Ching Yao$^{\rm *}$\footnote{Corresponding author. Email: yao@stat.sinica.edu.tw}}
\author{~~Daniel Wei-Chung Miao$^{\rm **}$}
\author{~~Xenos Chang-Shuo Lin$^{\rm ***}$\vspace*{5mm}}

\address{$^{\rm *}$\rm Institute of Statistical Science, Academia Sinica, Taipei 115, Taiwan} 
\address{$^{\rm **}$\rm National Taiwan University of Science and Technology, Taipei 106, Taiwan}
\address{$^{\rm ***}$\rm Aletheia University, 
New Taipei City 25103, Taiwan}

\end{frontmatter}

\vspace*{-1.8cm}
\begin{center}
\section*{Abstract}
\end{center}

The (conditional or unconditional) distribution of the continuous
scan statistic in a one-dimensional Poisson  process may be
approximated by that of a discrete analogue via time discretization
(to be referred to as the discrete approximation). With the help of
a change-of-measure argument, we derive the first-order term of the
discrete approximation which involves some functionals of the
Poisson process. Richardson's extrapolation is then applied to yield
a corrected (second-order) approximation. Numerical results are
presented to compare  various approximations.

\vspace*{5mm} {\noindent\bf Keywords:} Poisson process; Richardson's
extrapolation; Markov chain embedding; change of measure;
second-order approximation; stochastic geometry.

\vspace*{1cm}

\section{Introduction}

The subject of scan statistics in one dimension as well as in higher
dimensions has found a great many applications in diverse areas
ranging from astronomy to epidemiology, genetics and neuroscience.
See Glaz, Naus and Wallenstein \cite{gnw2001} and Glaz and Naus
\cite{gn2010} for a thorough review and comprehensive discussion of
scan distribution theory, methods and applications. See also Glaz,
Pozdnyakov and Wallenstein \cite{gpw2009} for
a collection of articles on recent developments.

In the one-dimensional setting, let $\Pi$ be a (homogeneous) Poisson point process of intensity $\lambda>0$
on the (normalized) unit interval $(0,1]$. For a specified window size $0<w<1$ and integers $N\ge k \ge 2$,
we are interested in finding the conditional and unconditional probabilities
\[
P(k;N,w):=\mathbb{P}(S_w\ge k \mid |\Pi|=N)\;\; \mathrm{and}\;\; P^\ast(k;\lambda,w):=\mathbb{P}(S_w\ge k),
\]
where $|\Pi|$ is the cardinality of the point set $\Pi$ (i.e. the total number of Poisson points) and
\[
S_w=S_w(\Pi):=\max_{0\le t \le 1-w} \left|\Pi \cap (t,t+w]\raisebox{4mm}{}\right|,
\]
the maximum number of Poisson points within any window of size $w$.
The (continuous) scan statistic $S_w$ arises from the likelihood
ratio test for  the null hypothesis $\mathcal{H}_0:$ the intensity
function $\lambda(t)=\lambda$ (constant) against the alternative
$\mathcal{H}_a:$ $\lambda(t)=\lambda+\Delta \mathbf{1}_{(a,
a+w]}(t)$ for (unknown) $0\le a\le 1-w$ and $\Delta>0$ where
$\mathbf{1}_{\mathcal{A}}$ denotes the indicator function of a set
$\mathcal{A}$.

By applying results on coincidence probabilities and the generalized
ballot problem ({\em cf.} Karlin and McGregor \cite{km1959} and
Barton and Mallows \cite{bm1965}), Huntington and Naus \cite{hn1975}
and Hwang \cite{hwang1977} derived closed-form expressions for
$P(k;N,w)$ which require to sum a large number of determinants of
large matrices and hence are in general not amenable  to numerical
evaluation. Later by exploiting the fact that $P(k;N,w)$ is
piecewise polynomial in $w$ with (finitely many) different
polynomials of $w$ in different ranges, Neff and Naus \cite{nn1980}
developed a more computationally feasible approach and presented
extensive tables for the {\em exact} $P(k;N,w)$ for various
combinations of $(k,N,w)$ with $N\le 25$. (More precisely, each
number in the tables has an error bounded by $10^{-9}$.) Noting that
$P^\ast(k;\lambda,w)$ is a weighted average of $P(k;N,w)$ over $N$
(with Poisson probabilities as weights), they also provided tables
for $P^\ast(k;\lambda,w)$ with $\lambda\le 16$ where the error size
for each tabulated number varies depending on the combination of
$(k,\lambda,w)$. (The errors tend to be greater for smaller values
of  $w$.) Huffer and Lin \cite{huffer1997, huffer1999} developed an
alternative approach (based on spacings) to computing the exact
$P(k;N,w)$.

Instead of finding the exact $P^\ast(k;\lambda,w)$, Naus
\cite{naus1982} proposed an accurate product-type approximation
based on a heuristic (approximate) Markov property while Janson
\cite{janson1984} derived some sharp bounds. See also Glaz and Naus
\cite{gn1991} for related results in a discrete setting. Treating
the problem as boundary crossing for a two-dimensional random field,
Loader \cite{loader1991} obtained effective large deviation
approximations for the tail probability of the scan statistic in one
and higher dimensions.  For more general large deviation
approximation results, see Siegmund and Yakir \cite{sy2000}, Chan
and Zhang \cite{cz2007} and Fang and Siegmund \cite{fs2015}.

\begin{figure}
\begin{center}
\includegraphics[width=0.75\textwidth]{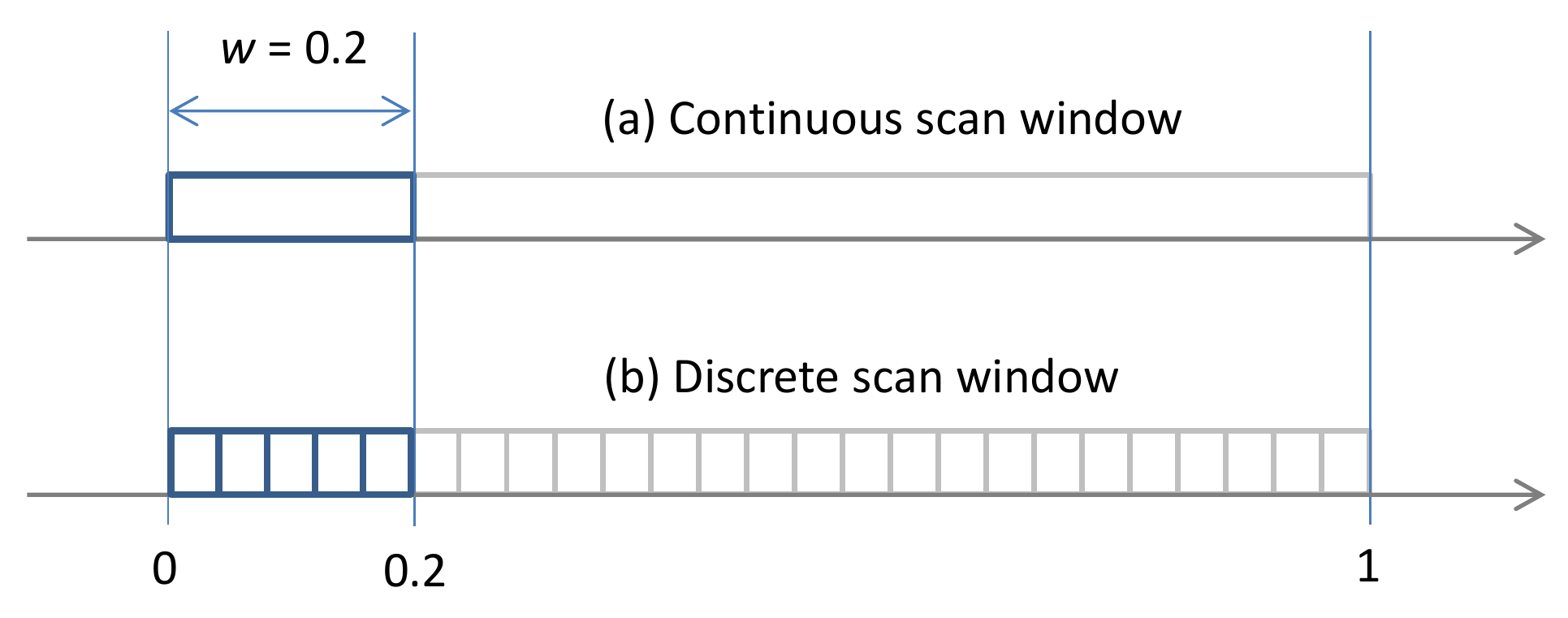}
\end{center}
\vspace*{-5mm} \caption{The continuous and discrete scan
windows}\label{fig:nested}
\end{figure}

The  continuous scan statistic $S_w$ may be approximated by a
discrete analogue via time discretization. Specifically, assuming
$w=p/q$ ($p,q$ relatively prime integers), partition the (time)
interval $(0,1]$ into  $n$ subintervals of length $n^{-1}$, $n$ a
multiple of $q$ ({\em cf.} Figure 1 with $w=1/5, n=25$). Each
subinterval (independently) contains  either no point (with
probability $1-\lambda/n$) or exactly one  point (with probability
$\lambda/n$). Since a window of size $w$ covers $nw$ subintervals,
as an approximation to $S_w$, we define the discrete scan statistic
$S_w^{(n)}$ to be the maximum number of points within any $nw$
consecutive subintervals. For large $n$,
$P^\ast(k;\lambda,w)=\mathbb{P}(S_w\ge k)$ may be approximated by
$\mathbb{P}(S_w^{(n)}\ge k)$, which can be readily calculated using
the Markov chain embedding method ({\em cf.} \cite{fu2001, fk1994,
ka1995}). Indeed, it is known that $\mathbb{P}(S_w\ge
k)-\mathbb{P}(S_w^{(n)}\ge k)=O(n^{-1})$ ({\em cf.}
\cite{fwl2012,wgf2013}).

In Section 2,  as $n$ (multiple of $q$) tends to infinity, we derive
the limit of $n[\mathbb{P}(S_w\ge k)-\mathbb{P}(S_w^{(n)}\ge k)]$,
which involves some functionals of $\Pi$. In order to establish this
limit result, we find it instructive to introduce a slightly
different discrete scan statistic (denoted $S'^{(n)}_w$) which is
stochastically smaller than $S_w$ and $S_w^{(n)}$. With a coupling
device, we derive the limits of  $n[\mathbb{P}(S_w\ge
k)-\mathbb{P}(S'^{(n)}_w\ge k)]$ and $n[\mathbb{P}(S_w^{(n)}\ge
k)-\mathbb{P}(S'^{(n)}_w\ge k)]$. In Section 3, using a
change-of-measure argument, a similar result is obtained for the
conditional probability $\mathbb{P}(S_w\ge k \mid |\Pi|=N)$. Based
on these limit  results, Richardson's extrapolation is then applied
to yield second-order approximations for the conditional and
unconditional distributions of the continuous scan statistic. In
Section 4, numerical results comparing the various approximations
are presented along with some discussion.

\section{The unconditional case}

Recall the window size $w=p/q$ with  $p$ and $q$ relatively prime
integers. For $n=mq\; (m=1,2,\dots)$, let $H^n_i, i=1,\dots, n$, be
$i.i.d.$ with $\mathbb{P}(H^n_i=0)=1-\lambda/n$ and
$\mathbb{P}(H^n_i=1)=\lambda/n$, and let $I^n_i, i=1,\dots, n$, be
$i.i.d.$ with $\mathbb{P}(I^n_i=0)=e^{-\lambda/n}$ and
$\mathbb{P}(I^n_i=1)=1-e^{-\lambda/n}$.  The $i.i.d.$  Bernoulli
sequence  $(H^n_1,\dots,H^n_n)$ approximates the Poisson point
process $\Pi$ by matching the expected number of points in each
subinterval, i.e.
$$\mathbb{E}(H^n_i)=\mathbb{E}\Big(\Big|\Pi \cap  \Big(\frac{i-1}{n},\frac{i}{n}\Big]\Big| \Big)=\frac{\lambda}{n}.$$
On the other hand, the $i.i.d.$  Bernoulli sequence
$(I^n_1,\dots,I^n_n)$  approximates $\Pi$ by matching the
probability of no point in each subinterval, i.e.
$$\mathbb{P}(I^n_i=0)=\mathbb{P}\Big(\Big| \Pi \cap \Big(\frac{i-1}{n},\frac{i}{n}\Big]\Big|=0\Big)=e^{-\lambda/n}.$$
The two discrete scan statistics $S^{(n)}_w$ and $S'^{(n)}_w$ are
now defined in terms of the two Bernoulli sequences as follows:
\begin{align*}
S^{(n)}_w&=S^{(n)}_{w,H}:=\max_{i=1\dots, n-nw+1} \sum_{r=i}^{i+nw-1} H^n_r\;,\\
S'^{(n)}_w&=S^{(n)}_{w,I}:=\max_{i=1\dots, n-nw+1}
\sum_{r=i}^{i+nw-1} I^n_r\;.
\end{align*}
Since $I^n_i$ is stochastically smaller than $H^n_i$ and $|\Pi \cap
((i-1)/n,i/n]|$, it follows that $S^{(n)}_{w,I}$ is stochastically
smaller than $S_w$ and $S^{(n)}_{w,H}$. In Sections 2.1 and 2.2, we
derive $\lim_{n\to \infty} n[\mathbb{P}(S_w \ge
k)-\mathbb{P}(S^{(n)}_{w,I}\ge k)]$ and  $\lim_{n\to \infty}
n[\mathbb{P}(S_w \ge k)-\mathbb{P}(S^{(n)}_{w,H}\ge k)]$,
respectively.

\subsection{Matching the probability of no point}

Since the Bernoulli sequence $(I^n_1,\dots,I^n_n)$ and $\Pi$ match
in the probability of no point in each subinterval, it is
instructive to define  $(I^n_1,\dots,I^n_n)$ in terms of $\Pi$ as
follows:
 $$I_i^n = \mathbf{1}\Big\{\Pi\cap \Big(\frac{i-1}{n},\frac{i}{n}\Big]\ne \emptyset\Big\},\;
i=1,\dots, n.$$ Thus, $(I^n_1,\dots,I^n_n)$ and $\Pi$ are defined on
the same probability space. In particular, $S_w\ge S_{w,I}^{(n)}$
with probability $1$. For fixed $w=p/q$ and for each (fixed)
$k=2,3,\dots$, let
\begin{align*}
\alpha&=\mathbb{P}(\mathcal{A}),\;\;\mbox{where}\;\;\mathcal{A}=\mathcal{A}_{k,w}:=\{S_w\ge k\},\\
\alpha_n&=\mathbb{P}(\mathcal{A}_n),\;\;\mbox{where}\;\;
 \mathcal{A}_n=\mathcal{A}_{n,k,w}:=\{S^{(n)}_{w,I}\ge k\}.
\end{align*}

Note that $\alpha=P^\ast(k;\lambda,w)$ defined in Section 1. In
order to derive the limit of $n(\alpha-\alpha_n)$ as $n\to\infty$,
we need to introduce some functionals of $\Pi$. Let $M:=|\Pi|$,
which is a Poisson random variable with mean $\lambda$. Writing $\Pi
= \{Q_1,\dots,Q_M\}$, assume (with probability $1$) that
$0<Q_1<\cdots<Q_M<1$. Further assume (with probability $1$) that
$w\notin \Pi, 1-w \notin \Pi$, and $Q_j\pm w\notin \Pi$ for
$j=1,\dots, M$ (i.e. $Q_j-Q_i\ne w$ for all $1\le i<j\le M$).
Define the functionals $\nu(\Pi)=\nu(\{Q_1,\dots,Q_M\})$ and
$\tilde{\nu}(\Pi)=\tilde{\nu}(\{Q_1,\dots,Q_M\})$ as follows:
\begin{align*}
\nu(\Pi):=&\sum_{\{\ell: Q_\ell<1-w\}}
\mathbf{1}\left\{\raisebox{5mm}{}\right. S_w<k,
\left|\Pi \cap (Q_\ell,Q_\ell+w]\raisebox{4mm}{}\right|=k-2,\\
&\mbox{~~~~~~~~~~~~~~~ $\left|\Pi \cap
(t,t+w]\raisebox{4mm}{}\right|\le k-2$ for all
$t$ with $Q_\ell\le t \le Q_\ell+w$}\left.\raisebox{5mm}{}\right\},\\
\tilde{\nu}(\Pi):=&\sum_{\ell=1}^M \mathbf{1}\left\{S_w<k,\;
 \max_{0\le t\le 1-w}
\left|(\Pi\cup \{Q_\ell\})\cap
(t,t+w]\raisebox{4mm}{}\right|=k\right\},
\end{align*}
where $\Pi\cup \{Q_\ell\}$ is interpreted as a multiset with
$Q_\ell$ having multiplicity $2$.
\begin{thm}
For $n = mq\; (m=1,2,\dots)$,
$$\lim_{n\to\infty}n(\alpha-\alpha_n)~=~\frac{\lambda}{2}~\mathbb{E}\left[\nu(\Pi)+\tilde{\nu}(\Pi)\raisebox{4mm}{}\right].$$
\end{thm}

\begin{proof}
Denoting the complement of $\mathcal{A}_n$ by $\mathcal{A}_n^c$ and
noting that $\mathcal{A}_n \subset \mathcal{A}$, we have
$\alpha-\alpha_n = \mathbb{P}(\mathcal{A}) -
\mathbb{P}(\mathcal{A}_n) = \mathbb{P}(\mathcal{A}\cap
\mathcal{A}_n^c)$. For $i=1,\dots,n$, let
$$\tilde{I}^n_i=\Big|\Pi\cap\Big(\frac{i-1}{n},\frac{i}{n}\Big]\Big|,\;
\mbox{the number of Poisson points in the $i$-th subinterval}.$$
 Then $\tilde{I}^n_i=0$ implies $I_i^n=0$ and
$\tilde{I}_i^n\ge 1$ implies $I_i^n = 1$. Consider the following
disjoint events
\begin{eqnarray}
\hspace*{11mm}\mathcal{G}_1\hspace*{1.8mm} &=& \{\tilde{I}_j^n\le 1, j=1,\dots,n\}, \nonumber\\
\hspace*{11mm}\mathcal{G}_{2,i} &=& \{\tilde{I}_i^n = 2, \tilde{I}_j^n\le 1\mbox{ for all $j\ne i$}\},~i=1,\dots,n, \nonumber\\
\hspace*{11mm}\mathcal{G}_3\hspace*{1.8mm} &=& \{\tilde{I}_j^n =
\tilde{I}_{j'}^n = 2\mbox{ for some $j\ne
j'$}\}\cup\{\tilde{I}_j^n\ge 3\mbox{ for some $j$}\}.\nonumber
\end{eqnarray}
We have
\begin{align}
\hspace*{11mm}\alpha -\alpha_n &~=& \mathbb{P}(\mathcal{A} \cap
\mathcal{A}_n^c)\hspace*{76mm}\nonumber\\
&~=& \mathbb{P}(\mathcal{A}\cap \mathcal{A}_n^c\cap\mathcal{G}_1) +
\sum_{i=1}^n\mathbb{P}(\mathcal{A}\cap
\mathcal{A}_n^c\cap\mathcal{G}_{2,i})+\mathbb{P}(\mathcal{A}\cap
\mathcal{A}_n^c\cap\mathcal{G}_3).\label{eq:eq_3_1}
\end{align}
Claim that
\begin{eqnarray}
\hspace*{11mm}\mathbb{P}(\mathcal{A}\cap \mathcal{A}_n^c\cap\mathcal{G}_1)\hspace*{1.8mm} &=& \frac{1}{2}\sum_{i=1}^{n-nw}P_i^{(n)} +~ O(n^{-2}),\label{eq:eq_3_2}\\
\hspace*{11mm}\sum_{i=1}^n\mathbb{P}(\mathcal{A}\cap
\mathcal{A}_n^c\cap\mathcal{G}_{2,i}) &=& \sum_{i=1}^{n}\widetilde{P}_i^{(n)} +~ O(n^{-2}),\label{eq:eq_3_3}\\
\hspace*{11mm}\mathbb{P}(\mathcal{A}\cap
\mathcal{A}_n^c\cap\mathcal{G}_3)\hspace*{1.8mm} &=&
O(n^{-2}),\label{eq:eq_3_4}
\end{eqnarray}
where
\begin{align}
P_i^{(n)} = \mathbb{P}\left(\mathcal{A}_n^c,
\sum_{r=i+1}^{i+nw-1}I_r^n = k-2,
I_i^n = I_{i+nw}^n = 1\right),~i=1,\dots,n-nw,\label{eq:eq_3_5}\\
\widetilde{P}_i^{(n)} = \mathbb{P}\left(\mathcal{A}_n^c, \tilde{I}_i^n = 2, \sum_{r=i^\prime}^{i^\prime+nw-1}I_r^n = k-1\mbox{ for some $i'$ with}\right.\hspace*{16mm}\label{eq:eq_3_6}\\
\hspace*{14mm}\left.1\le i^\prime\le i\le i^\prime + nw -1\le
n\raisebox{6mm}{}\right),~i=1,\dots,n.\nonumber
\end{align}

Since $\mathbb{P}(\mathcal{G}_3)=O(n^{-2})$, (\ref{eq:eq_3_4})
follows easily. To prove (\ref{eq:eq_3_2}), note that when
$\tilde{I}_i^n\le 1$ for all $i$ (i.e. on the event
$\mathcal{G}_1$), each subinterval $((i-1)/n,i/n]$ contains at most
one Poisson point. If $\tilde{I}_i^n = 1$, denote the only Poisson
point in $((i-1)/n,i/n]$ by $Q_{(i)}$ whose location is uniformly
distributed over $((i-1)/n,i/n]$. When $\tilde{I}_i^n\le 1$ for all
$i$, in order for $\mathcal{A}\cap \mathcal{A}_n^c$ to occur, there
must exist some pair $(i,i^\prime)$ with $i^\prime = i+nw$ such that
$$\sum_{r=i+1}^{i^\prime-1}\tilde{I}_r^n = k-2,\;
\tilde{I}_i^n = \tilde{I}_{i'}^{n} = 1,\;\mbox{and }Q_{(i^\prime)} -
Q_{(i)}<w.$$ So we have $\mathcal{A}\cap
\mathcal{A}_n^c\cap\mathcal{G}_1 =
\cup_{i=1}^{n-nw}\mathcal{G}_{1,i}$ where for $i = 1,\dots,n-nw$,
$$
\begin{array}{r}
\mathcal{G}_{1,i} = \mathcal{A}_n^c\cap\left\{\tilde{I}_j^n\le 1\mbox{ for all $j$, }\displaystyle\sum_{r=i+1}^{i+nw-1}\tilde{I}_r^n = k-2,\right.\\
\left.\displaystyle\tilde{I}_i^n = \tilde{I}_{i+nw}^{n} =
1,\;\mbox{and }Q_{(i+nw)} - Q_{(i)}<w\raisebox{5mm}{}\right\}.
\end{array}
$$
Since
$$\sum_{1\le i<j\le n-nw}\mathbb{P}(\tilde{I}_i^n = \tilde{I}_{i+nw}^n = \tilde{I}_j^n = \tilde{I}_{j+nw}^n = 1) = O(n^{-2}),$$
we have
\begin{eqnarray}
\mathbb{P}(\mathcal{A}\cap \mathcal{A}_n^c\cap\mathcal{G}_1) &=&
\sum_{i=1}^{n-nw}\mathbb{P}(\mathcal{G}_{1,i}) + O(n^{-2})\nonumber\\
&=& \frac{1}{2}\sum_{i=1}^{n-nw}\mathbb{P}(\mathcal{G}_{1,i}^\prime)
+ O(n^{-2}),\label{eq:eq_3_7}
\end{eqnarray}
where
$$\mathcal{G}_{1,i}^\prime = \mathcal{A}_n^c\cap\left\{\tilde{I}_j^n\le 1\mbox{ for all $j$, }\displaystyle\sum_{r=i+1}^{i+nw-1}\tilde{I}_r^n = k-2,\;\tilde{I}_i^n = \tilde{I}_{i+nw}^n = 1\right\}.$$
In (\ref{eq:eq_3_7}), we have used the facts that
$\tilde{I}_1^n,\dots,\tilde{I}_n^n$ are independent and that given
$\tilde{I}_i^n = \tilde{I}_{i+nw}^n = 1$, $Q_{(i)}$ and $Q_{(i+nw)}$
are (conditionally) independent and uniformly distributed over
$((i-1)/n, i/n]$ and $((i+nw-1)/n, (i+nw)/n]$, respectively, so that
$Q_{(i+nw)}-Q_{(i)}<w$ with (conditional) probability 1/2, which
implies $\mathbb{P}(\mathcal{G}_{1,i}) =
\frac{1}{2}\mathbb{P}(\mathcal{G}_{1,i}^\prime)$. For
$i=1,\dots,n-nw$, define
\begin{equation}
\mathcal{G}_{1,i}^{\prime\prime} =
\mathcal{A}_n^c\cap\left\{\sum_{r=i+1}^{i+nw-1}I_r^n = k-2,\;I_i^n =
I_{i+nw}^n = 1\right\},\label{GG}
\end{equation}
which is the event inside the parentheses on the right-hand side of
 (\ref{eq:eq_3_5}), so that $P_i^{(n)}=\mathbb{P}(\mathcal{G}_{1,i}^{\prime\prime})$. Note that $\mathcal{G}_{1,i}^\prime\subset\mathcal{G}_{1,i}^{\prime\prime}$ and
that
$\mathcal{G}_{1,i}^{\prime\prime}\setminus\mathcal{G}_{1,i}^\prime$
is contained in
$$\{I_i^n = I_{i+nw}^n = 1,\;\tilde{I}_j^n\ge 2\mbox{ for some $j$}\},$$
which has a probability of order $n^{-3}$. By (\ref{eq:eq_3_7}),
$$\mathbb{P}(\mathcal{A}\cap \mathcal{A}_n^c\cap\mathcal{G}_1) = \frac{1}{2}\sum_{i=1}^{n-nw}\mathbb{P}(\mathcal{G}_{1,i}^{\prime\prime}) + O(n^{-2}) = \frac{1}{2}\sum_{i=1}^{n-nw}P_i^{(n)} + O(n^{-2}),$$
establishing (\ref{eq:eq_3_2}).

To prove (\ref{eq:eq_3_3}), let $\mathcal{H} = \{I_j = I_{j+nw} =
1\mbox{ for some $1\le j \le n-nw$}\}$. On
$\mathcal{G}_{2,i}\cap\mathcal{H}^c$, in order for $\mathcal{A}\cap
\mathcal{A}_n^c$ to occur, there must exist some $i^\prime$ with
$1\le i'\le i\le i' + nw - 1\le n$ such that $\sum_{r=i'}^{i' +
nw-1}I_r^n = k-1$ (implying that $\sum_{r=i'}^{i' +
nw-1}\tilde{I}_r^n = k$). It follows that $\mathcal{A}\cap
\mathcal{A}_n^c\cap\mathcal{G}_{2,i}\cap\mathcal{H}^c\subset\mathcal{G}_{2,i}^\prime\subset
\mathcal{A}\cap \mathcal{A}_n^c\cap\mathcal{G}_{2,i}$, where
\begin{align*}
\mathcal{G}_{2,i}^\prime =
\mathcal{A}_n^c\cap\left\{\raisebox{5mm}{}\tilde{I}_i^n =
2,\;\tilde{I}_j^n\le 1\mbox{ for all $j\ne
i$},\right.\hspace*{50mm}\\
\hspace*{10mm}\left.\sum_{r=i'}^{i'+nw-1}I^n_r=k-1\mbox{ for some
$i^\prime$ with } 1\le i'\le i\le i'+nw-1\le n\right\}.
\end{align*}
Since $\mathbb{P}(\mathcal{G}_{2,i}\cap\mathcal{H}) = O(n^{-3})$, we
have $\sum_{i=1}^n\mathbb{P}(\mathcal{G}_{2,i}\cap\mathcal{H}) =
O(n^{-2})$, implying that
$$\sum_{i=1}^n\mathbb{P}(\mathcal{A}\cap \mathcal{A}_n^c\cap\mathcal{G}_{2,i}) = \sum_{i=1}^n\mathbb{P}(\mathcal{G}_{2,i}^\prime) + O(n^{-2}) =
\sum_{i=1}^n\mathbb{P}(\mathcal{G}_{2,i}^{\prime\prime}) +
O(n^{-2}),$$ where
\begin{align*}\mathcal{G}_{2,i}^{\prime\prime} = \mathcal{A}_n^c\cap\left\{\tilde{I}_i^n=2,
\sum_{r=i'}^{i'+nw-1}I_r^n =
k-1\right.\hspace*{40mm}\\\hspace*{10mm}\left.\mbox{ for some
$i^\prime$ with $1\le i'\le i\le i'+nw-1\le
n$}\raisebox{5mm}{}\right\}.
\end{align*}
(Note that
$\mathcal{G}_{2,i}^\prime\subset\mathcal{G}_{2,i}^{\prime\prime}$
and
$\mathcal{G}_{2,i}^{\prime\prime}\setminus\mathcal{G}_{2,i}^\prime$
is contained in the event $\{\tilde{I}_i^n=2, \tilde{I}_j^n\ge
2\mbox{ for some $j\ne i$}\}$, which has a probability of  order
$n^{-3}$.) By (\ref{eq:eq_3_6}), $\widetilde{P}_i^{(n)} =
\mathbb{P}(\mathcal{G}_{2,i}^{\prime\prime})$. This establishes
(\ref{eq:eq_3_3}).

By (\ref{eq:eq_3_1})--(\ref{eq:eq_3_4}), we have
\begin{equation}
\alpha - \alpha_n = \frac{1}{2}\sum_{i=1}^{n-nw}P_i^{(n)} +
\sum_{i=1}^{n}\widetilde{P}_i^{(n)} + O(n^{-2}).\label{eq:eq_3_8}
\end{equation}
For $i=1,\dots,n-nw$, let
$P_i^{\prime(n)}=\mathbb{P}(\mathcal{F}_i)$ where
\begin{align*}
 \mathcal{F}_i&:=\mathcal{A}_n^c \cap \Big\{
  \sum_{r=i+1}^{i+nw-1}I_r^n=k-2, I^n_i=1, I^n_{i+nw}=0,
 \;\mbox{sum of any $nw$}\notag\\
 &\mbox{~~~~~~~~~~~~
 consecutive $I^n_r$ including
 $r=i+nw$ is at most $k-2$}
 \Big\}.\label{extra}
\end{align*}
Claim that
\begin{equation}
P_i^{(n)}/P_i^{\prime(n)} = \rho_n \mbox{~~for all~~} i=1,\dots,
n-nw,\label{extra1}
\end{equation}
where
\begin{align*}
\rho_n:=\frac{\mathbb{P}(I_{i+nw}^n=1)}{~\mathbb{P}(I_{i+nw}^n=0)~}
~=~ \frac{~1-e^{-\lambda/n}~}{e^{-\lambda/n}} ~=~ e^{\lambda/n} - 1.
\end{align*}
To establish the claim, recall that
$P_i^{(n)}=\mathbb{P}(\mathcal{G}_{1,i}^{\prime\prime})$ where
$\mathcal{G}_{1,i}^{\prime\prime}$ ({\em cf.} (\ref{GG})) depends
only on $(I^n_1,\dots,I^n_n)$. It is instructive to interpret
$\mathcal{G}_{1,i}^{\prime\prime}$ as
 a collection of configurations
$(I^n_1,\dots,I^n_n)=(h_1,\dots,h_n)$ where $(h_1,\dots,h_n)$
satisfies $h_j=0$ or $1$ for all $j$, $h_i=h_{i+nw}=1$, and
$$\max_{j=1,\dots,n-nw+1} \sum_{r=j}^{j+nw-1} h_r <k\;,\;\; \sum_{r=i+1}^{i+nw-1} h_r=k-2.$$
Likewise, the event $\mathcal{F}_i$ is   a collection of
configurations $(I^n_1,\dots,I^n_n)=(h'_1,\dots,h'_n)$ where
$(h'_1,\dots,h'_n)$ satisfies $h'_j=0$ or $1$ for all $j$, $h'_i=1,
h'_{i+nw}=0$,
$$\max_{j=1,\dots,n-nw+1} \sum_{r=j}^{j+nw-1} h'_r <k\;,\;\; \sum_{r=i+1}^{i+nw-1} h'_r=k-2,$$
and {\em sum of any $nw$ consecutive $h'_r$ including $r=i+nw$ is at
most $k-2$}.
 It is readily seen that
a configuration $(I^n_1,\dots,I^n_n)=(h_1,\dots,h_n)$ is in
$\mathcal{G}_{1,i}^{\prime\prime}$ if and only if the configuration
$(I^n_1,\dots,I^n_n)=(h'_1,\dots,h'_n)$  is in $\mathcal{F}_i$ where
$(h'_1,\dots,h'_n)=(h_1,\dots,h_n)-\mathbf{e}_{i+nw}$ with
$\mathbf{e}_{i+nw}$ being the vector of zeroes except for  the
$(i+nw)$-th entry being $1$. The claim (\ref{extra1}) now follows
from the independence property of $I^n_1,\dots,I^n_n$.

By (\ref{extra1}),
\begin{align}
\rho^{-1}_n\sum_{i=1}^{n-nw}P_i^{(n)} =
\sum_{i=1}^{n-nw}P_i^{\prime(n)}& =\sum_{i=1}^{n-nw}
\mathbb{P}(\mathcal{F}_i)=
\mathbb{E}\left[{\nu}^{(n)}(\Pi)\right],\label{eq:eq_3_9}
\end{align}
where
\begin{align*}
{\nu}^{(n)}(\Pi):=&\sum_{i=1}^{n-nw}\mathbf{1}\left\{\mathcal{A}_n^c,  \sum_{r=i+1}^{i+nw-1}I_r^n=k-2, I^n_i=1, I^n_{i+nw}=0,\mbox{ sum of any}\right.\\
&\left.\mbox{~~~~~~~~$nw$ consecutive $I^n_r$ including $r=i+nw$ is
at most $k-2$}\raisebox{5mm}{}\right\}.
\end{align*}

To deal with $\widetilde{P}^{(n)}_i, i=1,\dots,n$, let
\begin{align*}
\widetilde{P}^{\prime(n)}_i:=\mathbb{P}\left(\mathcal{A}_n^c,
I_i^n=1,
\sum_{r=i^\prime}^{i^\prime+nw-1}I_r^n=k-1\;\mbox{ for some $i^\prime$ with}\right.\hspace*{10mm}\\
\left.1\le i^\prime\le i\le i^\prime+nw-1\le
n\raisebox{5mm}{}\right).
\end{align*}
By an argument similar to the proof of (\ref{extra1}), we have
$\widetilde{P}_i^{(n)}/\widetilde{P}_i^{\prime(n)} = \tilde{\rho}_n$
for all $i=1,\dots,n$ where
$$\tilde{\rho}_n ~=~ \frac{\mathbb{P}(\tilde{I}_i^n=2)}{~\mathbb{P}(I_i^n=1)~} ~=~ \frac{e^{-\lambda/n}(\lambda/n)^2}{~2(1-e^{-\lambda/n})~}.$$
So,
\begin{equation}\tilde{\rho}_n^{-1}\sum_{i=1}^n\widetilde{P}_i^{(n)}
~=~ \sum_{i=1}^n\widetilde{P}_i^{\prime(n)} =
\mathbb{E}\left[\tilde{\nu}^{(n)}(\Pi)\right],\label{eq:eq_3_10}\end{equation}
where
\begin{align*}
\tilde{\nu}^{(n)}(\Pi):=\sum_{i=1}^n\mathbf{1}\left\{\mathcal{A}_n^c,
I_i^n=1,
\sum_{r=i^\prime}^{i^\prime+nw-1}I_r^n=k-1\right.\hspace*{28mm}\\
\hspace*{10mm}\left.\mbox{ for some $i^\prime$ with }1\le
i^\prime\le i\le i^\prime+nw-1\le n\raisebox{6mm}{}\right\}.
\end{align*}
Since $\rho_n = \lambda/n + O(n^{-2})$ and $\tilde{\rho}_n =
\lambda/(2n) + O(n^{-2})$, it follows from (\ref{eq:eq_3_8}),
(\ref{eq:eq_3_9}) and (\ref{eq:eq_3_10}) that
\begin{equation}
n(\alpha-\alpha_n)
-\frac{\lambda}{2}\;\mathbb{E}\left[\nu^{(n)}(\Pi)+\tilde{\nu}^{(n)}(\Pi)\right]=O(n^{-1}).
\label{eq:eq_3_11}
\end{equation}
Note that $\nu^{(n)}(\Pi)$ and $\tilde{\nu}^{(n)}(\Pi)$ converge
a.s. to $\nu(\Pi)$ and $\tilde{\nu}(\Pi)$, respectively. Since
$$\max\{\nu^{(n)}(\Pi), \tilde{\nu}^{(n)}(\Pi)\}~\le~\sum_{i=1}^{n} \mathbf{1}\{I^n_i=1\} ~\le~ \left|\Pi\right|,$$ we have by
the dominated convergence theorem that
$\mathbb{E}[\nu^{(n)}(\Pi)+\tilde{\nu}^{(n)}(\Pi)]$ converges to
$\mathbb{E}[\nu(\Pi)+\tilde{\nu}(\Pi)]$, which together with
(\ref{eq:eq_3_11}) completes the proof.
\end{proof}

\begin{rem}
With a little more effort, it can be shown that
$$\mathbb{E}\left[\nu^{(n)}(\Pi)+\tilde{\nu}^{(n)}(\Pi)\right]-\mathbb{E}\left[\nu(\Pi)+\tilde{\nu}(\Pi)\raisebox{4mm}{}\right]=O(n^{-1}),$$
which together (\ref{eq:eq_3_11}) yields
\begin{equation}
\alpha-\alpha_n=C_\alpha n^{-1}+O(n^{-2}),\label{alpha}
\end{equation}
where
\begin{equation}
C_\alpha=\frac{\lambda}{2}\mathbb{E}\left[\nu(\Pi)+\tilde{\nu}(\Pi)\raisebox{4mm}{}\right].\label{alpha1}
\end{equation}
\end{rem}

\subsection{Matching the expected number of points}

Recall that $H_i^n, i = 1,\dots,n$ are {\it i.i.d.} with
$\mathbb{P}(H_i^n=0) = 1 - \lambda/n$ and $\mathbb{P}(H_i^n=1) =
\lambda/n$. Let $\beta_n = \mathbb{P}(\mathcal{B}_n)$ where
$$\mathcal{B}_n =\mathcal{B}_{n,k,w}:= \{S^{(n)}_{w,H}\ge k\}=\left\{\max_{i=1,\dots,n-nw+1}\sum_{r=i}^{i+nw-1}H_r^n~\ge~ k\right\}.$$
\begin{lem} For $n=mq\; (m=1,2,\dots)$,
$$\lim_{n\to\infty}\frac{2n}{\lambda^2}(\beta_n - \alpha_n) = -\alpha + \int_0^1
\mathbb{P}\left(\max_{0\le t\le
1-w}\left|(\Pi\cup\{u\})\cap(t,t+w]\raisebox{5mm}{}\right|\ge
k\right)du.$$
\end{lem}
\begin{proof} Let $L_i^n, i = 1,\dots,n$ be {\it i.i.d.} and independent
of $I_1^n,\dots,I_{n}^n$ such that $\mathbb{P}(L_i^n=0) = (1 -
\lambda/n)e^{\lambda/n} = 1-\mathbb{P}(L_i^n=1)$. Letting
$\tilde{L}_i^n = \max\{I_i^n, L_i^n\}$ and noting that
$\mathbb{P}(\tilde{L}_i^n=0) = \mathbb{P}(I_i^n=0\mbox{ and
}L_i^n=0) = 1 - \lambda/n = \mathbb{P}(H_i^n=0)$, we have
$\mathcal{L}(\tilde{L}_1^n,\dots,\tilde{L}_{n}^n)
=\mathcal{L}(H_1^n,\dots,H_{n}^n)$
 where $\mathcal{L}(\mathbf{V})$ denotes the law of a  random vector $\mathbf{V}$,
so that $\beta_n=\mathbb{P}(\mathcal{B}_n) =
\mathbb{P}(\widetilde{\mathcal{B}}_n)$ where
$$\widetilde{\mathcal{B}}_n = \left\{\max_{i=1,\dots,n-nw+1}\sum_{r=i}^{i+nw-1}\tilde{L}_r^n~\ge~ k\right\}.$$
Since $I_i^n = 1$ implies $\tilde{L}_i^n = 1$, we have
$\mathcal{A}_n\subset\widetilde{\mathcal{B}}_n$. Letting $S_n =
\sum_{i=1}^{n}L_i^n$ and noting that
$\widetilde{\mathcal{B}}_n\cap\{S_n=0\}=\mathcal{A}_n\cap\{S_n=0\}$
and that
$$\mathbb{P}(S_n=0) = 1-\frac{\lambda^2}{2n} +
O(n^{-2}),~\mathbb{P}(S_n=1) = \frac{\lambda^2}{2n} +
O(n^{-2}),~\mathbb{P}(S_n\ge 2) = O(n^{-2}),$$ we have
\begin{align}
\beta_n = \mathbb{P}(\widetilde{\mathcal{B}}_n) =&
~\mathbb{P}(\widetilde{\mathcal{B}}_n|S_n=0)\mathbb{P}(S_n=0) +
\mathbb{P}(\widetilde{\mathcal{B}}_n|S_n=1)\mathbb{P}(S_n=1)\nonumber\\
& ~+ \mathbb{P}(\widetilde{\mathcal{B}}_n|S_n\ge 2)\mathbb{P}(S_n\ge
2)\nonumber\\ =&
~\mathbb{P}(\mathcal{A}_n|S_n=0)\left(1-\frac{\lambda^2}{2n}\right)
+ \mathbb{P}(\widetilde{\mathcal{B}}_n|S_n=1)\frac{\lambda^2}{2n} +
O(n^{-2})\nonumber\\
 =&
~\alpha_n\left(1-\frac{\lambda^2}{2n}\right) +
\mathbb{P}(\widetilde{\mathcal{B}}_n|S_n=1)\frac{\lambda^2}{2n} +
O(n^{-2}).\label{eq:eq_3_12}
\end{align}
Claim that
\begin{equation}
\lim_{n\to \infty}
\mathbb{P}(\widetilde{\mathcal{B}}_n|S_n=1)=\displaystyle\int_0^1
\mathbb{P}\left(\max_{0\le
 t\le 1-w}\Big|(\Pi\cup\{u\})\cap(t,t+w] \Big|\ge k \right) du,\label{new}
 \end{equation}
 which together with (\ref{eq:eq_3_12}) yields the desired result.

It remains to establish the claim (\ref{new}). Let $Q$ be a random
point which is uniformly distributed on $(0,1]$ and independent of
$\Pi$. Let
 $$\hat{I}^n_i= \mathbf{1}\Big\{(\Pi \cup \{Q\}) \cap \Big(\frac{i-1}{n},\frac{i}{n}\Big]\ne \emptyset\Big\},\;\; i=1,\dots,n.$$
It is readily seen that
$\mathcal{L}(\tilde{L}^n_1,\dots,\tilde{L}^n_n\mid
S_n=1)=\mathcal{L}(\hat{I}^n_1,\dots,\hat{I}^n_n)$, which implies
$\mathbb{P}(\widetilde{\mathcal{B}}_n|S_n=1)=\mathbb{P}(\hat{\mathcal{B}}_n),$
where
$$\hat{\mathcal{B}}_n= \left\{\max_{i=1,\dots,n-nw+1}\sum_{r=i}^{i+nw-1}\hat{I}_r^n~\ge~ k\right\}.$$
Since $\mathbf{1}_{\hat{\mathcal{B}}_n}$ converges a.s. to
$\mathbf{1}\{\max_{0\le t \le 1-w} |(\Pi \cup \{Q\})\cap (t,t+w]|
\ge k\}$, we have
\begin{align*}
\lim_{n\to \infty} \mathbb{P}(\hat{\mathcal{B}}_n)&=\mathbb{P}\Big(\max_{0\le t \le 1-w} \Big|(\Pi \cup \{Q\})\cap (t,t+w]\Big| \ge k\Big)\\
&=\displaystyle\int_0^1 \mathbb{P}\left(\max_{0\le
 t\le 1-w}\Big|(\Pi\cup\{u\})\cap(t,t+w] \Big|\ge k \right) du,
 \end{align*}
the claim (\ref{new}) follows. This completes the proof of the
lemma.
\end{proof}

\begin{thm}
\begin{flalign}\lim_{n\to\infty}\frac{2n}{\lambda^2}(\alpha -
\beta_n) ~=~ &
\displaystyle\frac{1}{\lambda}\mathbb{E}\left[\nu(\Pi)+\tilde{\nu}(\Pi)\raisebox{4mm}{}\right]+\alpha\vspace*{1mm}\nonumber\\
& \hspace*{0mm}-\displaystyle\int_0^1 \mathbb{P}\left(\max_{0\le
t\le 1-w}\left|(\Pi\cup\{u\})\cap(t,t+w]\raisebox{5mm}{}\right|\ge
k\right)du.\nonumber
\end{flalign}
\end{thm}
\begin{proof} Note that
$$\frac{2n}{\lambda^2}(\alpha -
\beta_n) = \frac{2n}{\lambda^2}(\alpha - \alpha_n) -
\frac{2n}{\lambda^2}(\beta_{n} - \alpha_{n}),$$ which together with
Theorem 2.1 and Lemma 2.2 yields the desired result.
\end{proof}

\begin{rem}
Similarly to (\ref{alpha}), it can be shown that
\begin{equation}
\alpha-\beta_n=C_\beta n^{-1}+O(n^{-2}),\label{beta}
\end{equation}
where
\begin{align}
C_\beta&=\frac{1}{2}\lambda^2 \alpha+\frac{\lambda}{2}\mathbb{E}\left[\nu(\Pi)+\tilde{\nu}(\Pi)\raisebox{4mm}{}\right]\notag\\
&\;\;\;\;\;-\frac{\lambda^2}{2}\int_0^1 \mathbb{P}\left(\max_{0\le
t\le 1-w}\Big|(\Pi\cup\{u\})\cap(t,t+w]\Big|\ge
k\right)du.\label{beta1}
\end{align}
\end{rem}

\section{The conditional case}

Following the notation of Section 2, $\Pi = \{Q_1,\dots,Q_M\}$ is a
Poisson point process of intensity $\lambda$ on $(0,1]$, where $M$
is a Poisson random variable with mean $\lambda$. For given $N\ge
k=2,3,\dots$, we are interested in approximating
$$\gamma^{(N)}:=P(k;N,w) = \mathbb{P}\left(\left.\max_{0\le t\le 1-w}\left|\Pi\cap(t,t+w]\raisebox{4mm}{}\right|\ge k\right|M=N\right).$$
Conditional on $M=N$, the $N$ points $0<Q_1<\dots<Q_N<1$ are the
order statistics of $N$ independent and uniformly distributed random
variables on $(0,1]$. Denote by $\Pi^N$ a set of $N$ {\it i.i.d.}
uniform random variables on $(0,1]$. Then
$\mathcal{L}(\Pi^N)=\mathcal{L}(\Pi | M=N)$ and $\gamma^{(N)} =
\mathbb{P}(\mathcal{E}^N)$ where
$$\mathcal{E}^N=\mathcal{E}^N_{k,w}: = \left\{\max_{0\le t\le
1-w}\left|\Pi^N\cap(t,t+w]\raisebox{4mm}{}\right|\ge k\right\}.$$ As
in Section 2, with $n= mq$ ($m=1,2,\dots$), the interval $(0,1]$ is
partitioned into $n$ subintervals of length $n^{-1}$, so that a
window of size $w=p/q$ covers $nw$ subintervals. As an approximation
to $N$ points uniformly distributed on $(0,1]$, we randomly select
$N$ of the $n$ subintervals and assign a point to each of them. Let
$J_i^n = 1$ or 0 according to whether or not the $i$-th subinterval
is selected (so as to contain a point). Then $\sum_{i=1}^{n}J_i^n =
N$ and for $h_i = 0$ or 1,
$$\mathbb{P}_N(J_i^n = h_i, i = 1,\dots,n) = \left\{
    \begin{array}{ll}
      1{\mbox{\LARGE /}}{\mbox{\large ${n\choose N}$}}, & \hbox{if~ $\sum_{i=1}^{n}h_i = N$,} \\
      0, & \hbox{otherwise,}
    \end{array}
  \right.
$$
where the subscript $N$ in $\mathbb{P}_N$ signifies that there are
$N$ 1's in $J_1^n,\dots,J_{n}^n$.  While in Section 2,
$(I_1^n,\dots,I_{n}^n)$  is defined in terms of $\Pi$ in order to
make use of a coupling argument, there is no natural way to define
$(J_1^n,\dots,J_{n}^n)$ and $\Pi^N$ on the same probability space.
As no danger of confusion may arise, we will use the same
probability measure notation $\mathbb{P}_N$ for both the probability
space where $\Pi^N$ is defined and the probability space where
$(J^n_1,\dots,J^n_n)$ is defined. Let
\begin{equation}
\gamma_n^{(N)} =
\mathbb{P}_N(\mathcal{E}_n^N)\;,\;\;\mbox{where}\;\; \mathcal{E}_n^N
:= \left\{\max_{i=1,\dots,n-nw+1}\sum_{r=i}^{i+nw-1}J_r^n~\ge~
k\right\}.\label{eq:eq_4_0}
\end{equation}

\begin{thm}
For $N$ fixed and $n = mq\;(m=1,2,\dots)$,
\begin{eqnarray}\lim_{n\to\infty}n(\gamma^{(N)}-\gamma_n^{(N)})
&=&\frac{1}{2}N(N-1)(\gamma^{(N-1)}-\gamma^{(N)})
\nonumber\\
&&\hspace*{10mm}
+~\frac{1}{2}N\mathbb{E}\left[\nu(\Pi)+\tilde{\nu}(\Pi)\left|M=N-1\raisebox{4mm}{}\right.\right].\nonumber
\end{eqnarray}
\end{thm}
\begin{proof} For notational simplicity, the superscript $N$ in
$\mathcal{E}^N$ and $\mathcal{E}_n^N$ is suppressed while to avoid
possible confusion, $\mathbb{P}_N$ is not abbreviated to
$\mathbb{P}$ as later a change-of-measure argument requires
consideration of $\mathbb{P}_{N-1}$. Let $\tilde{J}_i =
\left|\Pi^N\cap((i-1)/N,i/N]\raisebox{3mm}{}\right|$, $i=1,\dots,n$,
and define the (disjoint)  events
$$
\begin{array}{l}U_1 = \{\tilde{J}_j^n\le 1, j = 1,\dots,n\},\vspace*{2mm}\\
U_2 = \bigcup_{i=1}^n U_{2,i},\; U_{2,i} = \{\tilde{J}_i^n=2, \tilde{J}_j^n\le 1\mbox{ for all $j\ne i$}\},\vspace*{2mm}\\
U_3 = \{\tilde{J}_j^n = \tilde{J}_{j'}^n = 2 \mbox{ for some $j\ne
j'$}\}\cup\{\tilde{J}_j^n\ge 3\mbox{ for some $j$}\}.
\end{array}
$$
We have
\begin{eqnarray}
\mathbb{P}_N(U_1) &=&
\displaystyle\prod_{i=1}^N\left(1-\frac{i-1}{n}\right) ~=~
1 - \frac{N(N-1)}{2n} + O(n^{-2}),\label{eq:eq_4_1}\vspace*{1mm}\\
\mathbb{P}_N(U_2) &=& \displaystyle{N\choose
2}n^{-1}\prod_{j=1}^{N-2}\left(1-\frac{j}{n}\right) ~=~
\frac{N(N-1)}{2n} + O(n^{-2}),\label{U2}
\end{eqnarray}
and $\mathbb{P}_N(U_3) = O(n^{-2})$, so that
\begin{eqnarray}
\gamma^{(N)} &=& \mathbb{P}_N(\mathcal{E}^N) =
\mathbb{P}_N(\mathcal{E})\nonumber\vspace*{2mm}\\ &=&
\mathbb{P}_N(\mathcal{E}\cap
U_1) + \mathbb{P}_N(\mathcal{E}\cap U_2) + \mathbb{P}_N(\mathcal{E}\cap U_3)\nonumber\vspace*{2mm}\\
&=& \mathbb{P}_N(\mathcal{E}|U_1)\mathbb{P}_N(U_1) +
\mathbb{P}_N(\mathcal{E}\cap U_2) + O(n^{-2}).\label{eq:eq_4_2}
\end{eqnarray}
We first work on $\mathbb{P}_N(\mathcal{E}|U_1)$. Write $\mathcal{E}
= \widetilde{\mathcal{E}}_n\cup (\mathcal{E}\cap
\widetilde{\mathcal{E}}_n^c)$ where
$$
\begin{array}{l}\widetilde{\mathcal{E}}_n := \displaystyle\left\{\max_{i = 1,\dots,n-nw+1}\sum_{r=i}^{i+nw-1}\tilde{J}_r^n\ge k\right\}~(\subset \mathcal{E}).
\end{array}
$$
Note that given $U_1$, the conditional distribution of
$(\tilde{J}_1^n,\dots,\tilde{J}_n^n)$ is equal to the distribution
of $(J_1^n,\dots,J_n^n)$ (i.e.
$\mathcal{L}(\tilde{J}_1^n,\dots,\tilde{J}_n^n | U_1) =
\mathcal{L}(J_1^n,\dots,J_n^n))$, and that
$\widetilde{\mathcal{E}}_n$ depends on
$(\tilde{J}_1^n,\dots,\tilde{J}_n^n)$ in the same way that
$\mathcal{E}_n=\mathcal{E}^N_n$ does on $(J_1^n,\dots,J_n^n)$ ({\em
cf.} (\ref{eq:eq_4_0})). So we have
$\mathbb{P}_N(\widetilde{\mathcal{E}}_n|U_1) =
\mathbb{P}_N(\mathcal{E}_n) = \gamma_n^{(N)}$, and
\begin{flalign}
\hspace*{5mm}\mathbb{P}_N(\mathcal{E}|U_1) ~=~ &
\mathbb{P}_N(\widetilde{\mathcal{E}}_n|U_1) +
\mathbb{P}_N(\mathcal{E}\cap \widetilde{\mathcal{E}}_n^c)|U_{1})\nonumber\\
=~ & \gamma_n^{(N)} +
\mathbb{P}_N(\mathcal{E}\cap\widetilde{\mathcal{E}}_n^c\;|U_1).\label{eq:eq_4_3}
\end{flalign}
If $\tilde{J}_i^n = 1$, denote the only point of $\Pi^N$ in
$((i-1)/n,i/n]$ by $Q_{(i)}$, whose location is uniformly
distributed over $((i-1)/n,i/n]$. When $\tilde{J}_i^n\le 1$ for all
$i$ (i.e.  on the event $U_1$), in order for
\begin{eqnarray}
\mathcal{E}\cap\widetilde{\mathcal{E}}_n^c ~=~ \left\{\max_{0\le
t\le
1-w}\left|\Pi^N\cap(t,t+w]\raisebox{4mm}{}\right|\ge k\right\}\hspace*{30mm}\nonumber\vspace*{1mm}\\
\hspace*{30mm}\cap\left\{\max_{i=0,1,\dots,n-nw}\left|\Pi^N\cap\left(\frac{i}{n},
\frac{i}{n}+w\right]\raisebox{4mm}{}\right|<
k\right\}\nonumber\end{eqnarray} to occur, there must exist some
pair $(i,i')$ with $i^\prime = i + nw$ such that
$\sum_{r=i+1}^{i^\prime-1}\tilde{J}_r^n = k-2$, $\tilde{J}_i^n =
\tilde{J}_{i^\prime}^n = 1$, and $Q_{(i^\prime)}-Q_{(i)}<w$. So we
have
$$\mathcal{E}\cap\widetilde{\mathcal{E}}_n^c \cap U_1 = \bigcup_{i=1}^{n-nw}U_{1,i}$$
where for $i=1,\dots,n-nw$,
\begin{eqnarray}
U_{1,i} ~=~  \widetilde{\mathcal{E}}_n^c\cap\left\{\tilde{J}_j^n\le 1\mbox{ for all $j$, }\sum_{r=i+1}^{i+nw-1}\tilde{J}_r^n = k-2,\right.\hspace*{10mm}\nonumber\vspace*{1mm}\\
\hspace*{30mm}\left.\tilde{J}_i^n = \tilde{J}_{i+nw}^n = 1,\;
Q_{(i+nw)}-Q_{(i)}<w\raisebox{6mm}{}\right\}.\nonumber\end{eqnarray}
Since
\begin{eqnarray}
&\displaystyle\sum_{1\le i<j\le n-nw}\mathbb{P}_N(U_{1,i}\cap
U_{1,j})\hspace*{55mm}\nonumber\vspace*{1mm}\\
\le&\displaystyle\sum_{1\le i<j\le n-nw}\mathbb{P}_N(\tilde{J}_i^n =
\tilde{J}_{i+nw}^n = \tilde{J}_j^n = \tilde{J}_{j+nw}^n = 1) =
O(n^{-2}),\nonumber
\end{eqnarray}
we have
\begin{eqnarray}
\mathbb{P}_N(\mathcal{E}\cap\widetilde{\mathcal{E}}_n^c\mid U_1)& ~=~ &  \sum_{i=1}^{n-nw}\mathbb{P}_N(U_{1,i}|U_1) + O(n^{-2})\nonumber\\
& ~=~ & \frac{1}{2}\sum_{i=1}^{n-nw}\mathbb{P}_N(U_{1,i}^\prime|U_1)
+ O(n^{-2}),\label{eq:eq_4_4}
\end{eqnarray}
where for $i=1,\dots,n-nw$,
$$U_{1,i}^\prime =  \widetilde{\mathcal{E}}_n^c\cap\left\{\tilde{J}_j^n\le 1\mbox{ for all $j$, }\sum_{r=i+1}^{i+nw-1}\tilde{J}_r^n =
k-2,\;\tilde{J}_i^n = \tilde{J}_{i+nw}^n = 1\right\}.$$ In
(\ref{eq:eq_4_4}), we have used the fact that for any given $h_j =
0$ or 1 ($j = 1,\dots,n$) with $\sum_{j=1}^n h_j = N$ and $h_i =
h_{i+nw} = 1$, conditional on $\tilde{J}_j^n = h_j$, $j =
1,\dots,n$, $Q_{(i)}$ and $Q_{(i+nw)}$ are independent and uniformly
distributed over $((i-1)/n,i/n]$ and $((i+nw-1)/n,(i+nw)/n]$,
respectively, so that $Q_{(i+nw)}-Q_{(i)}<w$ with probability $1/2$,
which implies $\mathbb{P}_N(U_{1,i}|U_1) =
\frac{1}{2}\mathbb{P}_N(U_{1,i}^\prime|U_1)$.

Note that $U_{1,i}^\prime, i = 1,\dots,n$ depend only on
$\tilde{J}_1^n,\dots,\tilde{J}_n^n$. Since
$\mathcal{L}(\tilde{J}_1^n,\dots,\tilde{J}_n^n|$ $U_1)$ $=
\mathcal{L}(J_1^n,\dots,J_n^n)$, we have
\begin{equation}
\mathbb{P}_N(U_{1,i}^\prime|U_1) = \mathbb{P}_N(V_i),\; i=
1,\dots,n-nw,\label{eq:eq_4_5}
\end{equation} where
$$V_i = \mathcal{E}_n^c\cap\left\{\sum_{r=i+1}^{i+nw-1}J_r^n = k-2, J_i^n = J_{i+nw}^n = 1\right\}.$$
(Note that $V_i$ depends on $(J_1^n,\dots,J_n^n)$ in the same way
that $U_{1,i}^\prime$ does on
$(\tilde{J}_1^n,\dots,\tilde{J}_n^n)$.)

 We will simplify $\sum_{i=1}^{n-nw} \mathbb{P}_N(U_{1,i}^\prime|U_1)=\sum_{i=1}^{n-nw} \mathbb{P}_N(V_i)$
 via a change-of measure argument. It is instructive to interpret the event $V_i$ as a collection
of configurations $(J_1^n,\dots, J_n^n)=(h_1,\dots,h_n)$ where
$(h_1,\dots,h_n)$ satisfies $h_r = 0$ or $1$ for $r=1,\dots,n$, and
$$\sum_{r=1}^{n}h_r = N,
\max_{j=1,\dots,n-nw+1}\sum_{r=j}^{j+nw-1}h_r<k,~
\sum_{r=i+1}^{i+nw-1}h_r = k-2,~h_i = h_{i+nw} = 1.$$ Let
$$
\begin{array}{ll}V_i^\ast =\left\{\displaystyle\sum_{r=1}^n J_r^n=N-1,\max_{j=1,\dots,n-nw+1}\sum_{r=j}^{j+nw-1}J_r^n<k, \right.\nonumber\vspace*{1mm}\\
\hspace*{15mm}\displaystyle\sum_{r=i+1}^{i+nw-1}J_r^n = k-2, J_i^n =
1, J_{i+nw}^n = 0,\mbox{ sum of any
$nw$}\vspace*{0mm}\\\hspace*{17mm}\left. \mbox{consecutive $J_r^n$
including $r = i+nw$ is at most
$k-2$}\raisebox{6mm}{}\right\},\nonumber
\end{array}
$$
which is interpreted as a collection of configurations
$(J_1^n,\dots, J_n^n)=(h_1^\ast,\dots,h_n^\ast)$, where
$(h_1^\ast,\dots,h_n^\ast)$ satisfies $h_r^\ast = 0$ or 1 for
$r=1,\dots, n$, and
$$
\begin{array}{l}
\displaystyle\sum_{r=1}^{n}h_r^\ast = N-1,
\max_{j=1,\dots,n-nw+1}\sum_{r=j}^{j+nw-1}h_r^\ast < k,
\sum_{r=i+1}^{i+nw-1}h_r^\ast = k-2,\vspace*{2mm}\\
\mbox{$h_i^\ast = 1, h_{i+nw}^\ast = 0$, and sum of any $nw$
consecutive $h_r^\ast$ including $r$}\vspace*{2mm}\\
\mbox{$= i+nw$ is at most $k-2$.}
\end{array}$$
If a configuration $(J_1^n,\dots, J_n^n)=(h_1,\dots,h_n)$ is in
$V_i$, then the configuration $(J_1^n,\dots,
J_n^n)=(h_1^\ast,\dots,h_n^\ast)$ is in $V_i^\ast$ provided
$h_r^\ast = h_r$ for all $r\ne i+nw$ and $h_{i+nw} = 1$,
$h_{i+nw}^\ast = 0$. In other words, a configuration is in $V_i$ if
and only if with the $(i+nw)$-th entry replaced by 0, it is in
$V_i^\ast$. Note that the number of nonzero entries for a
configuration in $V_i^\ast$ equals $N-1$. Recall that the notation
$\mathbb{P}_N$ ($\mathbb{P}_{N-1}$, {\it resp.}) denotes the
probability measure for $(J_1^n,\dots,J_{n}^n)$ with
$\sum_{r=1}^{n}J_r^n = N$ ($\sum_{r=1}^{n}J_r^n = N-1$, {\it
resp.}). It follows that
$$\frac{\mathbb{P}_N(V_i)}{~\mathbb{P}_{N-1}(V_i^\ast)~} ~=~ \frac{~{n\choose N-1}~}{{n\choose N}} ~=~ \frac{N}{n-N+1}.$$
Therefore,
\begin{flalign}
\sum_{i=1}^{n-nw}\mathbb{P}_N(V_i) ~=~ &\frac{N}{n-N+1}\sum_{i=1}^{n-nw}\mathbb{P}_{N-1}(V_i^\ast)\nonumber\\
=~ &
\frac{N}{n-N+1}\mathbb{E}_{N-1}\left[\nu_1^{(n)}(J_1^n,\dots,J_{n}^n)\raisebox{4mm}{}\right],\label{eq:eq_4_6}
\end{flalign}
where $$
\begin{array}{l} \nu_1^{(n)}(J_1^n,\dots,J_{n}^n) =
\displaystyle\sum_{i=1}^{n-nw}\mathbf{1}\left\{\displaystyle\max_{j=1,\dots,n-nw+1}\sum_{r=j}^{j+nw-1}J_r^n<k,\sum_{r=i+1}^{i+nw-1}J_r^n
= k-2,\right.\vspace*{2mm}\\  \hspace*{42mm}J_i^n = 1, J_{i+nw}^n
= 0,\mbox{ sum of any $nw$ consecutive $J_r^n$}\vspace*{1mm}\\
\hspace*{42mm}\left.\mbox{including $r = i+nw$ is at most
$k-2$}\raisebox{6mm}{}\right\}.
\end{array}
$$
By (\ref{eq:eq_4_3})--(\ref{eq:eq_4_6}),
\begin{flalign}
\mathbb{P}_N(\mathcal{E}|U_1) ~=~ &\gamma_n^{(N)} + \frac{1}{2}\frac{N}{~n-N+1~}\mathbb{E}_{N-1}\left[\nu_1^{(n)}(J_1^n,\dots,J_{n}^n)\raisebox{4mm}{}\right] + O(n^{-2})\nonumber\vspace*{1mm}\\
=~ & \gamma_n^{(N)} + \frac{N}{~2n~}\mathbb{E}\left[\nu(\Pi)\left|M
= N-1\raisebox{4mm}{}\right.\right] + o(n^{-1}),\label{eq:eq_4_7}
\end{flalign}
since
$\displaystyle\lim_{n\to\infty}\mathbb{E}_{N-1}\left[\nu_1^{(n)}(J_1^n,\dots,J_{n}^n)\raisebox{4mm}{}\right]
=
\mathbb{E}\left[\nu(\Pi)\left|M=N-1\raisebox{4mm}{}\right.\right]$.

Next, we deal with $\mathbb{P}_N(\mathcal{E}\cap U_2)$, the second
term on the right-hand side of (\ref{eq:eq_4_2}). Recall that $U_2$
is the event that exactly one of $\tilde{J}_1^n,\dots,\tilde{J}_n^n$
equals 2 and the others are all less than 2. For
$(\mathcal{E}\setminus\widetilde{\mathcal{E}}_n)\cap U_2$ to occur,
there must exist either some $(i,i^\prime)$ with $i^\prime=i+nw$ and
$(\tilde{J}^n_i,\tilde{J}^n_{i^\prime}) \in \{(1,2),(2,1)\}$ or
$(i,i^\prime,i^{\prime\prime})$ with $i^\prime=i+nw$ and
$(\tilde{J}^n_i,\tilde{J}^n_{i^\prime},\tilde{J}^n_{i^{\prime\prime}})=(1,1,2)$.
 This implies that
$\mathbb{P}_N((\mathcal{E}\setminus\widetilde{\mathcal{E}}_n)\cap
U_2) = O(n^{-2})$, so that
\begin{equation}
\mathbb{P}_N(\mathcal{E}\cap U_2) =
\mathbb{P}_N(\widetilde{\mathcal{E}}_n\cap U_2) + O(n^{-2}).
\label{eq:eq_4_8}
\end{equation}
Again we interpret $U_2$ as a collection of configurations
$(\tilde{J}_1^n,\dots,\tilde{J}_n^n)=( h_1^\prime,\dots,h_n^\prime)$
where $\sum_{j=1}^{n}h_j^\prime = N$ and all $h_i^\prime\le 1$
except for one $h_i^\prime = 2$. Fix $(h_1,\dots,h_n)$ with all $h_i
=0$ or 1 and $\sum_{i=1}^n h_i = N-1$, which is considered as a
configuration for $(J_1^n,\dots,J_n^n)$ (under $\mathbb{P}_{N-1}$).
This configuration corresponds to $N-1$ configurations in $U_2$ by
replacing one of the $N-1$ $h_i = 1$ with $h_i^\prime = 2$. The
probability of each of the latter $N-1$ configurations for
$(\tilde{J}_1^n,\dots,\tilde{J}_n^n)$ (under $\mathbb{P}_N$) equals
$$\frac{N!}{~2n^N~} ~=~ \frac{\xi_n^N}{~{n\choose N-1}~} ~=~ \xi_n^N\mathbb{P}_{N-1}(J_i^n=h_i,i=1,\dots,n),$$
where
\begin{equation}
\xi_n^N ~=~ \frac{~N!{n\choose N-1}~}{~2n^N~} ~=~ \frac{N}{~2n~} +
O(n^{-2}). \label{eq:eq_4_9}
\end{equation}
Thus for the fixed configuration
$(J_1^n,\dots,J_n^n)=(h_1,\dots,h_n)$ (under $\mathbb{P}_{N-1}$),
among the corresponding $N-1$ configurations for
$(\tilde{J}_1^n,\dots,\tilde{J}_n^n)$ (under $\mathbb{P}_{N}$),
 the sum of the probabilities of those  in $\widetilde{\mathcal{E}}_n\cap U_2$ equals
\begin{equation}
\xi_n^N\nu_2^{(n)}(h_1,\dots,h_{n})\mathbb{P}_{N-1}(J_i^n = h_i,
i=1,\dots,n) \label{eq:eq_4_10}
\end{equation}
where
\begin{flalign}
\nu_2^{(n)}(h_1,\dots,h_{n}) ~=~ &
(N-1)\mathbf{1}\left\{\max_{i=1,\dots,n-nw+1}\sum_{r=i}^{i+nw-1}h_r\ge
k\right\}\nonumber\\
&+~\sum_{i=1}^{n}\mathbf{1}\left\{\max_{i=1,\dots,n-nw+1}\sum_{r=i}^{i+nw-1}h_r<k,\;
h_i=1,
\right.\nonumber\\
&\hspace*{20mm}\left. \max_{i^\prime =
i-nw+1,\dots,i}\sum_{r=i^\prime}^{i^\prime + nw-1}h_r =
k-1\right\},\nonumber
\end{flalign}
with the convention that $h_r=0$ for $r<0$ or $r>n$. It follows from
(\ref{eq:eq_4_9}) and (\ref{eq:eq_4_10}) that
\begin{flalign}
&\mathbb{P}_N(\widetilde{\mathcal{E}}_n\cap U_2) ~=~
\xi_n^N\mathbb{E}_{N-1}\left[\nu_2^{(n)}(J_1^n,\dots,J_{n}^n)\raisebox{4mm}{}\right]\nonumber\\~=~&
\frac{N}{2n}\left((N-1)\gamma^{(N-1)} +
\mathbb{E}\left[\tilde{\nu}(\Pi)\left|M=N-1\raisebox{4mm}{}\right.\right]\raisebox{4.5mm}{}\right)
+ o(n^{-1}),\label{eq:eq_4_11}
\end{flalign}
since
$$\lim_{n\to\infty}\mathbb{E}_{N-1}\left[\nu_2^{(n)}(J_1^n,\dots,J_{n}^n)\raisebox{4mm}{}\right]
= (N-1)\gamma^{(N-1)} + \mathbb{E}\left[\tilde{\nu}(\Pi)\left|M =
N-1\raisebox{4mm}{}\right.\right].$$ Finally, by (\ref{eq:eq_4_1}),
(\ref{eq:eq_4_2}), (\ref{eq:eq_4_7}), (\ref{eq:eq_4_8}) and
(\ref{eq:eq_4_11}),
\begin{flalign}
\gamma^{(N)} ~=~& \mathbb{P}_N(\mathcal{E})\nonumber\\
~=~& \mathbb{P}_N(\mathcal{E}|U_1)\mathbb{P}(U_1) +
\mathbb{P}_N(\mathcal{E}\cap U_2) + O(n^{-2})\nonumber\\
~=~& \left(\gamma_n^{(N)} + \frac{N}{2n}\mathbb{E}\left[\nu(\Pi)\left|M=N-1\raisebox{4mm}{}\right.\right]\right)\left(1-\frac{N(N-1)}{2n}\right)\nonumber\vspace*{1mm}\\
& ~~+~
\frac{N}{2n}\left((N-1)\gamma^{(N-1)}+\mathbb{E}\left[\tilde{\nu}(\Pi)\left|M=N-1\raisebox{4mm}{}\right.\right]\raisebox{4.5mm}{}\right)
+ o(n^{-1}),\nonumber
\end{flalign}
from which the theorem follows. \end{proof}

\begin{rem}
Similarly to (\ref{alpha}) and (\ref{beta}), it can
be shown that
\begin{equation}
\gamma^{(N)}-\gamma_n^{(N)}=C_\gamma n^{-1}+O(n^{-2}),\label{gamma}
\end{equation}
where
\begin{equation}
C_\gamma =\frac{N(N-1)}{2}(\gamma^{(N-1)}-\gamma^{(N)})
+\frac{N}{2}\mathbb{E}\left[\nu(\Pi)+\tilde{\nu}(\Pi) \;\Big|
\;M=N-1\right].\label{gamma1}
\end{equation}
\end{rem}

\begin{rem}
Note that $\alpha_n$ and $\beta_n$ are weighted
averages of $\gamma^{(N)}_n$ over $N$ with binomial probabilities
$\binom{n}{N} p_n^N (1-p_n)^{n-N}$ as weights where
$p_n=1-e^{-\lambda/n}$ for $\alpha_n$ and $p_n=\lambda/n$ for
$\beta_n$. The limits  $\lim_{n\to \infty} n(\alpha -\alpha_n)$ and
$\lim_{n \to \infty} n(\alpha-\beta_n)$ in Theorems 2.1 and 2.3 can
be {\em formally} derived from  $\lim_{n \to \infty} n(\gamma^{(N)}
-\gamma^{(N)}_n)$ by interchanging $\lim_n$ and $\Sigma_N$. While
the details are omitted, it is of interest to note that the formal
derivations suggest the following identity
$$\int_0^1 \mathbb{P}\left(\max_{0\le t\le 1-w}\Big|(\Pi\cup\{u\})\cap(t,t+w]\Big|\ge k\right)du=\sum_{N=0}^\infty \frac{e^{-\lambda} \lambda^N}{N!}\gamma^{(N+1)},$$
which can be proved by observing that both sides are equal to
 $$\mathbb{P}\left(\max_{0\le t\le 1-w}\Big|(\Pi\cup\{Q\})\cap(t,t+w]\Big|\ge k\right)$$
where $Q$ is a random point which is uniformly distributed  on
$(0,1]$ and independent of $\Pi$.
\end{rem}

\section{Numerical results and discussion}

Using the Markov chain embedding method ({\em cf.}
\cite{fu2001, fwl2012,ka1995}), we computed the discrete
approximations $\alpha_n,\beta_n$ and $\gamma^{(N)}_n$ for various
combinations of parameter values $(k,w,\lambda)$ (the unconditional
case) and $(k,w,N)$ (the conditional case). Figure 2 plots
$n(\alpha-\alpha_n), n(\alpha-\beta_n)$ and
$n(\gamma^{(N)}-\gamma^{(N)}_n)$  for $n=25(5)600$ with
$k=5,w=0.4,\lambda=8$ and $N=8$, while Table 1 presents the  values for $n=50, 100(100)600$, where
the superscript $(N)$ in $\gamma^{(N)}$ and $\gamma^{(N)}_n$ is suppressed for ease of notation. The exact probabilities
$\alpha=P^\ast(k;\lambda,w)=P^\ast(5;8,0.4)=0.628144085$ and
$\gamma^{(8)}=P(k;N,w)=P(5;8,0.4)=0.780861440$ are taken from \cite{nn1980}. By Theorems 2.1, 2.3 and 3.1,  $n(\alpha-\alpha_n), n(\alpha-\beta_n)$ and $n(\gamma^{(N)}-\gamma^{(N)}_n)$ converge, respectively,  to the limits $C_\alpha, C_\beta$ and $C^{(N)}_\gamma$ which are given in
(\ref{alpha1}), (\ref{beta1}) and (\ref{gamma1}). These limits were estimated by Monte Carlo simulation with $10^6$ replications, resulting in
$C_\alpha=4.6322\pm0.0096, C_\beta=0.8297\pm0.0167, C^{(N)}_\gamma=2.7279\pm 0.0114$.

In view of (\ref{alpha}), (\ref{beta}) and (\ref{gamma}), the rate
of convergence for $\alpha_n,\beta_n$ and $\gamma^{(N)}_n$ can be
improved by using Richardson's extrapolation. Specifically for
$w=p/q$, suppose $n$ is even such that $n/2$ is a multiple of $q$.
Let
$$\tilde{\alpha}_n:=2\alpha_n -\alpha_{n/2}\;,\;\; \;\tilde{\beta}_n:=2\beta_n-\beta_{n/2}\;,\;\;\; \tilde{\gamma}_n^{(N)}:=2 \gamma_n^{(N)}-\gamma_{n/2}^{(N)}\;.$$ Then we have
$$\alpha-\tilde{\alpha}_n=O(n^{-2})\;, \;\;\;\alpha-\tilde{\beta}_n=O(n^{-2})\;,\;\;\; \gamma^{(N)}-\tilde{\gamma}^{(N)}_{n}=O(n^{-2})\;.$$
 Table 2 presents numerical results
comparing $\alpha_n, \tilde{\alpha}_n, \beta_n$ and
$\tilde{\beta}_n$ for the unconditional case. Table 3 compares
$\gamma^{(N)}_n$ and $\tilde{\gamma}^{(N)}_n$ for the conditional
case.

\begin{figure}[p]
\begin{center}
\includegraphics[width=.75\textwidth]{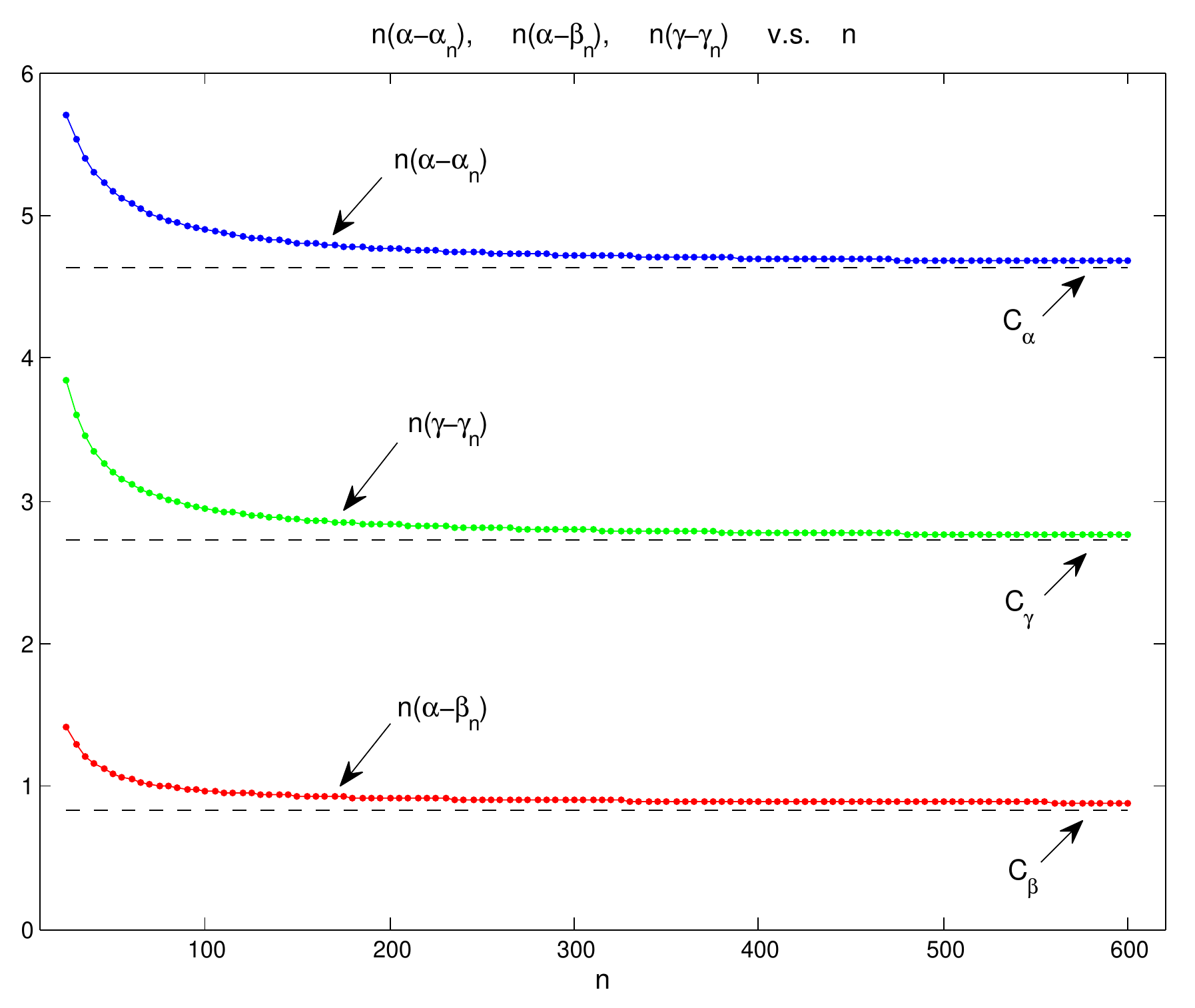}
\caption{Plot of $n(\alpha - \alpha_n)$, $n(\alpha - \beta_n)$, and
$n(\gamma - \gamma_n)$ for $n = 25(5)600$ with parameters $w = 0.4,
k = 5, \lambda = 8, N = 8$.}\label{fig:convergence}
\end{center}
\end{figure}

\begin{table*}[p]
\caption{Exact values for  $n(\alpha-\alpha_n)$, $n(\alpha-\beta_n)$
and $n(\gamma -\gamma_n)$ and the estimated limits $\mbox{with $k=5,
w=0.4, \lambda=8, N=8$}$} \label{tbl:constant_C}
\begin{center}
\resizebox{.7\textwidth}{!}{
\begin{tabular}{crrr}
\hline
\multirow{2}{*}{\hspace*{5mm}$n$\hspace*{5mm}} & \multicolumn{2}{c}{Unconditional\raisebox{4mm}{}} & \multicolumn{1}{c}{Conditional}\\
\cmidrule(lr){2-3}\cmidrule(lr){4-4} &
\multicolumn{1}{c}{$n(\alpha-\alpha_n)$} &
\multicolumn{1}{c}{$n(\alpha-\beta_n)$} &
\multicolumn{1}{c}{$n(\gamma-\gamma_n)$}  \\
\hline \hspace*{1.5mm}$50$ & ~~5.168687891~~
 & ~~1.088138412~~ & ~~3.203626718~~  \\
$100$ &    ~~4.898446228~~ & ~~0.969210822~~ & ~~2.948216315~~  \\
$200$ & ~~4.764816517~~ & ~~0.917535676~~ & ~~2.832944959~~  \\
$300$ & ~~4.720634969~~ & ~~0.901298885~~ & ~~2.796179050~~  \\
$400$ & ~~4.698621375~~ & ~~0.893353201~~ & ~~2.778092954~~  \\
$500$ & ~~4.685438621~~ & ~~0.888639155~~ & ~~2.767334580~~  \\
$600$ & ~~4.676660628~~ & ~~0.885518032~~ & ~~2.760200812~~
\\\hline
Limit\raisebox{3.5mm}{} & \multicolumn{1}{c}{$C_\alpha$} &
\multicolumn{1}{c}{$C_\beta$} & \multicolumn{1}{c}{$C_\gamma$} \\
\hline {Estimate}\raisebox{3mm}{} & \multicolumn{1}{c}{4.6322} &
\multicolumn{1}{c}{0.8297} &
\multicolumn{1}{c}{2.7279}\\
{(Std. Err.)} & \multicolumn{1}{c}{(0.0096)} &
\multicolumn{1}{c}{(0.0167)} & \multicolumn{1}{c}{(0.0114)}
\\ \hline
\end{tabular}
}
\end{center}
\end{table*}

\begin{table}[t]
\begin{center}
\caption{The unconditional
case}\label{tbl:unconditional}\vspace*{-4mm}
\resizebox{1.0\textwidth}{!}{ {\footnotesize
\begin{tabular}{cccrcccccc}
\hline
\multicolumn{3}{c}{\footnotesize Parameters$\raisebox{4mm}{}$}  & \hspace*{8mm} &  &  &  $n$  &  &  & {\footnotesize ~~~~~~Exact~~~~~~} \\
\cmidrule(lr){1-3}\cmidrule(lr){5-9} \cmidrule(lr){10-10}
\hspace*{3mm}$\lambda$\hspace*{3mm}    &  \hspace*{3mm}$w$\hspace*{3mm}     &   \hspace*{3mm}$k$\hspace*{3mm}    &   \hspace*{18mm}\raisebox{3mm}{}                         &   $25$        &   $50$ &  $100$   &  $200$    &  $400$    &  \hspace*{10mm}$\alpha$\hspace*{10mm}\\
\hline 4   &  0.2  &  3    &   $\alpha_n$\hspace*{3mm} & 0.226474137
& 0.297081029  &  0.330413369  &  0.346549002  & 0.354481473
     &  ~0.362322986\raisebox{4mm}{}  \\
    &       &       &   $\alpha - \alpha_n$\hspace*{3mm}         &   0.135848849  &   0.065241957   &  0.031909617   &  0.015773984   &  0.007841513
     &    \\
    &       &       &   $\widetilde{\alpha}_n$\hspace*{3mm}      &        &   0.367687921  &  0.363745709  &  0.362684635  &  0.362413943    &    \\
    &       &       &   $\alpha - \widetilde{\alpha}_n$\hspace*{3mm}  &   &   $-$0.005364935\hspace*{2.6mm} & $-$0.001422723\hspace*{2.6mm}  & $-$0.000361649\hspace*{2.6mm}  & $-$0.000090957\hspace*{2.6mm}   &    \\
\cmidrule(l){4-9}
    &       &       &   $\beta_n$\hspace*{3mm}                   &   0.269466265  &  0.321289109  &  0.342964036  &  0.352912780  &  0.357682871    &    \\
    &       &       &   $\alpha - \beta_n$\hspace*{3mm}             &   0.092856721  &  0.041033877  &  0.019358950  &  0.009410206  &  0.004640115    &    \\
    &       &       &   $\widetilde{\beta}_n$\hspace*{3mm}       &    & 0.373111952  &  0.364638964  &  0.362861523  &  0.362452963   &    \\
    &       &       &   $\alpha - \widetilde{\beta}_n$\hspace*{3mm} &    & $-$0.010788966\hspace*{2.6mm}  &  $-$0.002315978\hspace*{2.6mm} &  $-$0.000538537\hspace*{2.6mm} &  $-$0.000129977\hspace*{2.6mm}  &    \\
\cmidrule(lr){1-10} 4   &  0.2  &  4    &   $\alpha_n$\hspace*{3mm}
&   0.028528199  &  0.063252847  &  0.083861016 & 0.094813938 &
0.100432989    &  ~0.106139839\raisebox{3mm}{}  \\
    &       &       &   $\alpha - \alpha_n$\hspace*{3mm}                       &   0.077611640  &   0.042886992  &  0.022278823  &  0.011325901  &  0.005706850     &    \\
    &       &       &   $\widetilde{\alpha}_n$\hspace*{3mm}      &       &   0.097977495  &   0.104469184   &  0.105766860  &   0.106052039    &    \\
    &       &       &   $\alpha - \widetilde{\alpha}_n$\hspace*{3mm}           &   & 0.008162344   &  0.001670655   &  0.000372979  &   0.000087800    &    \\
\cmidrule(l){4-9}
    &       &       &   $\beta_n$\hspace*{3mm}                   &   0.037826080  &  0.071921990  &  0.089167692  &  0.097701122  &  0.101933685    &    \\
    &       &       &   $\alpha - \beta_n$\hspace*{3mm}          & 0.068313759  &   0.034217849  &  0.016972147  &  0.008438717  &  0.004206154    &    \\
    &       &       &   $\widetilde{\beta}_n$\hspace*{3mm}       &  & 0.106017899  &   0.106413395  &   0.106234551  &   0.106166248    &    \\
    &       &       &   $\alpha - \widetilde{\beta}_n$\hspace*{3mm} &    & 0.000121940  &  $-$0.000273556\hspace*{2.6mm} &  $-$0.000094712\hspace*{2.6mm} &  $-$0.000026409\hspace*{2.6mm}  &    \\
\cmidrule(lr){1-10} 8   &  0.4  &  5    &   $\alpha_n$\hspace*{3mm}
& 0.400190890  &   0.524770327   &  0.579159623   &  0.604320002 &  0.616397532 &  ~0.628144085\raisebox{3mm}{}  \\
    &       &       &   $\alpha - \alpha_n$\hspace*{3mm}   & 0.227953195  &  0.103373758  &  0.048984462  &  0.023824083     &  0.011746553
    &    \\
    &       &       &   $\widetilde{\alpha}_n$\hspace*{3mm}      &   &  0.649349765 & 0.633548918  &  0.629480382  &  0.628475061   &    \\
    &       &       &   $\alpha - \widetilde{\alpha}_n$\hspace*{3mm}     &      &   $-$0.021205680\hspace*{2.6mm}  &   $-$0.005404833\hspace*{2.6mm} & $-$0.001336297\hspace*{2.6mm}  & $-$0.000330976\hspace*{2.6mm}  &      \\
\cmidrule(l){4-9}
    &       &       &   $\beta_n$\hspace*{3mm}                   &   0.571524668  &  0.606381317  &  0.618451977  &  0.623556407  &  0.625910702    &    \\
    &       &       &   $\alpha - \beta_n$\hspace*{3mm}          &   0.056619417  & 0.021762768  &  0.009692108  &  0.004587678  &  0.002233383    &    \\
    &       &       &   $\widetilde{\beta}_n$\hspace*{3mm}       &   & 0.641237966 & 0.630522637  &  0.628660836  &  0.628264997   &    \\
    &       &       &   $\alpha - \widetilde{\beta}_n$\hspace*{3mm} &   & $-$0.013093881\hspace*{2.6mm}  &  $-$0.002378552\hspace*{2.6mm} &  $-$0.000516751\hspace*{2.6mm} &  $-$0.000120912\hspace*{2.6mm}  &    \\
\cmidrule(lr){1-10} 8   &  0.4  &  6    &   $\alpha_n$\hspace*{3mm} &   0.156407681  &  0.278520053  &  0.341202440  &  0.372097133  &  0.387351968   &  ~0.402452588\raisebox{3mm}{}  \\
    &       &       &   $\alpha - \alpha_n$\hspace*{3mm}         &   0.246044907  &  0.123932535  &  0.061250148  &  0.030355455     &  0.015100620     &    \\
    &       &       &   $\widetilde{\alpha}_n$\hspace*{3mm}      &                &  0.400632426  &  0.403884826  &  0.402991826     &  0.402606803     &    \\
    &       &       &   $\alpha - \widetilde{\alpha}_n$\hspace*{3mm}  &           &  0.001820162  &   $-$0.001432238\hspace*{2.6mm} & $-$0.000539238\hspace*{2.6mm}  & $-$0.000154215\hspace*{2.6mm}  & \\
\cmidrule(l){4-9}
    &       &       &   $\beta_n$\hspace*{3mm}                   &   0.278663391  &  0.351874806  &  0.379351117  &  0.391387631  &  0.397034846    &    \\
    &       &       &   $\alpha - \beta_n$\hspace*{3mm}          &   0.123789197  &  0.050577782  &  0.023101471  &  0.011064957  &  0.005417742    &    \\
    &       &       &   $\widetilde{\beta}_n$\hspace*{3mm}       &                &  0.425086221  &  0.406827428  &  0.403424144  &  0.402682062    &    \\
    &       &       &   $\alpha - \widetilde{\beta}_n$\hspace*{3mm} &             & $-$0.022633633\hspace*{2.6mm}  &  $-$0.004374840\hspace*{2.6mm} &  $-$0.000971556\hspace*{2.6mm} &  $-$0.000229474\hspace*{2.6mm}  &    \\

\hline

\end{tabular}
} }
\vspace*{2mm}
\end{center}
\end{table}

\begin{table}[t]
\begin{center}
\caption{The conditional case}\label{tbl:conditional}\vspace*{-4mm}
\resizebox{1.0\textwidth}{!}{ {\footnotesize
\begin{tabular}{cccrcccccc}
\hline
\multicolumn{3}{c}{\footnotesize Parameters$\raisebox{4mm}{}$}  & \hspace*{8mm} &  &  &  $n$  &  &  & {\footnotesize ~~~~~~Exact~~~~~~}\\ 
\cmidrule(lr){1-3}\cmidrule(lr){5-9}\cmidrule(lr){10-10}
\hspace*{3mm}$w$\hspace*{3mm}    &  \hspace*{3mm}$k$\hspace*{3mm}     &   \hspace*{3mm}$N$\hspace*{3mm}    &   \hspace*{18mm}\raisebox{3mm}{}                         &   $25$        &   $50$ &  $100$   &  $200$    &  $400$    &  \hspace*{10mm}$\gamma$\hspace*{10mm}\\
\hline 0.2 &  4  &  6    &   $\gamma_n$\hspace*{3mm} & 0.080688876 &
 0.155913836  &  0.194799457  &  0.214242757  & 0.223935622
    &  ~0.233600000\raisebox{4mm}{}  \\
    &       &       &   $\gamma- \gamma_n$\hspace*{3mm}         &   0.152911124  &  0.077686164  &  0.038800543  &   0.019357243   &  0.009664378     &    \\
    &       &       &   $\widetilde{\gamma}_n$\hspace*{3mm}      &     &   0.231138796  &  0.233685077  &  0.233686058  &  0.233628487    &    \\
    &       &       &   $\gamma - \widetilde{\gamma}_n$\hspace*{3mm}  &   &   0.002461204 & $-$0.000085077\hspace*{2.6mm}  & $-$0.000086058\hspace*{2.6mm}  & $-$0.000028487\hspace*{2.6mm}   &    \\
\cmidrule(l){4-10}
    &       &  7    &   $\gamma_n$\hspace*{3mm}                   &   0.166798419  &   0.294914521  &   0.354660825  &   0.383179030  &   0.397084766    &  ~0.410752000  \\
    &       &       &   $\gamma - \gamma_n$\hspace*{3mm}             &   0.243953581  &   0.115837479  &   0.056091175  &   0.027572970  &   0.013667234    &    \\
    &       &       &   $\widetilde{\gamma}_n$\hspace*{3mm}       &    & 0.423030623  &   0.414407129  &   0.411697234  &   0.410990502    &    \\
    &       &       &   $\gamma - \widetilde{\gamma}_n$\hspace*{3mm} &    & $-$0.012278623\hspace*{2.6mm}  &  $-$0.003655129\hspace*{2.6mm} &  $-$0.000945234\hspace*{2.6mm} &  $-$0.000238502\hspace*{2.6mm}  &    \\
\cmidrule(lr){4-10}
    &       &  8    &   $\gamma_n$\hspace*{3mm}                  &   0.291588655  &  0.469102180  &  0.542215920  &  0.575193126  &  0.590840920     &  ~0.605949440 \raisebox{3mm}{}  \\
    &       &       &   $\gamma - \gamma_n$\hspace*{3mm}         &   0.314360785  &  0.136847260  &  0.063733520  &  0.030756314  &  0.015108520     &    \\
    &       &       &   $\widetilde{\gamma}_n$\hspace*{3mm}      &                &  0.646615704  &  0.615329660  &  0.608170332  &  0.606488715    &    \\
    &       &       &   $\gamma - \widetilde{\gamma}_n$\hspace*{3mm}  &    &   $-$0.040666264\hspace*{2.6mm} & $-$0.009380220\hspace*{2.6mm}  & $-$0.002220892\hspace*{2.6mm}  & $-$0.000539275\hspace*{2.6mm}   &    \\
\cmidrule(l){4-10}
    &       &  9    &   $\gamma_n$\hspace*{3mm}                   &   0.448718168  &  0.651907101  &  0.723414803  &  0.753448974  &  0.767220941    &   ~0.780225536 \\
    &       &       &   $\gamma - \gamma_n$\hspace*{3mm}          &   0.331507368  &  0.128318435  &  0.056810733  &  0.026776562  &  0.013004595    &    \\
    &       &       &   $\widetilde{\gamma}_n$\hspace*{3mm}       &     & 0.855096034 & 0.794922506  &  0.783483145  &  0.780992907    &    \\
    &       &       &   $\gamma - \widetilde{\gamma}_n$\hspace*{3mm} &    & $-$0.074870498\hspace*{2.6mm}  &  $-$0.014696970\hspace*{2.6mm} &  $-$0.003257609\hspace*{2.6mm} &  $-$0.000767371\hspace*{2.6mm}  &    \\
\cmidrule(lr){1-10} 0.4 &   5   &  6    &   $\gamma_n$\hspace*{3mm}
&   0.162450593  &   0.224402953   &  0.254838093 & 0.269838436 & 0.277276942    &  ~0.284672000\raisebox{3mm}{}  \\
    &       &       &   $\gamma - \gamma_n$\hspace*{3mm}                       &   0.122221407  &  0.060269047  &  0.029833907  &  0.014833564  &  0.007395058     &    \\
    &       &       &   $\widetilde{\gamma}_n$\hspace*{3mm}      &             &   0.286355312 & 0.285273233  &  0.284838778  &  0.284715449    &    \\
    &       &       &   $\gamma - \widetilde{\gamma}_n$\hspace*{3mm}           &     &   $-$0.001683312\hspace*{2.6mm} & $-$0.000601233\hspace*{2.6mm}  & $-$0.000166778\hspace*{2.6mm}  & $-$0.000043449\hspace*{2.6mm}   &    \\
\cmidrule(l){4-10}
    &       &  7    &   $\gamma_n$\hspace*{3mm}                   &   0.371395881  &  0.463718058  &  0.504028994  &  0.522865538  &  0.531971853    &  ~0.540876800  \\
    &       &       &   $\gamma - \gamma_n$\hspace*{3mm}          &   0.169480919  &  0.077158742  &  0.036847806  &  0.018011262  &  0.008904947    &    \\
    &       &       &   $\widetilde{\gamma}_n$\hspace*{3mm}       &     & 0.556040235 & 0.544339929  &  0.541702083  &  0.541078168    &    \\
    &       &       &   $\gamma - \widetilde{\gamma}_n$\hspace*{3mm} &  & $-$0.015163435\hspace*{2.6mm}  &  $-$0.003463129\hspace*{2.6mm} &  $-$0.000825283\hspace*{2.6mm} &  $-$0.000201368\hspace*{2.6mm}  &    \\
\cmidrule(lr){4-10}
    &       &  8    &   $\gamma_n$\hspace*{3mm}                  &   0.627251924  &  0.716788906  &  0.751379277  &  0.766696715  &  0.773916208     &  ~0.780861440\raisebox{3mm}{}  \\
    &       &       &   $\gamma - \gamma_n$\hspace*{3mm}         &   0.153609516  &  0.064072534  &  0.029482163  &  0.014164725  &  0.006945232     &    \\
    &       &       &   $\widetilde{\gamma}_n$\hspace*{3mm}      &                &   0.806325887 & 0.785969648  &  0.782014154  &  0.781135700    &    \\
    &       &       &   $\gamma - \widetilde{\gamma}_n$\hspace*{3mm}              &     &   $-$0.025464447\hspace*{2.6mm} & $-$0.005108208\hspace*{2.6mm}  & $-$0.001152714\hspace*{2.6mm}  & $-$0.000274260\hspace*{2.6mm}   &    \\
\cmidrule(l){4-10}
    &       &  9    &   $\gamma_n$\hspace*{3mm}                   &   0.864220071  &  0.918826852  &  0.936992617  &  0.944511915  &  0.947942424    &  ~0.951173120  \\
    &       &       &   $\gamma - \gamma_n$\hspace*{3mm}          &   0.086953049  &  0.032346268  &  0.014180503  &  0.006661205  &  0.003230696    &    \\
    &       &       &   $\widetilde{\gamma}_n$\hspace*{3mm}       &     & 0.973433633 & 0.955158381  &  0.952031214  &  0.951372933    &    \\
    &       &       &   $\gamma - \widetilde{\gamma}_n$\hspace*{3mm} &    & $-$0.022260513\hspace*{2.6mm}  &  $-$0.003985261\hspace*{2.6mm} &  $-$0.000858094\hspace*{2.6mm} &  $-$0.000199813\hspace*{2.6mm}  &    \\

\hline

\end{tabular}
} }

\vspace*{2mm}
\end{center}
\end{table}

\begin{rem}
In Tables 1--3, we have taken relatively large
values of  $w=0.2$ and $0.4$ since the {\em exact} unconditional
probabilities reported in \cite{nn1980} are less accurate for
$w<0.2$. Figure 1 shows that  $n(\alpha-\alpha_n),
n(\alpha-\beta_n)$ and $n(\gamma^{(N)}-\gamma^{(N)}_n)$
monotonically approach $C_\alpha, C_\beta$ and $C^{(N)}_\gamma$,
respectively.
 In Table 2, $\beta_n$ is consistently more accurate than $\alpha_n$, which is not surprising since
$\alpha_n < \min\{\alpha, \beta_n\}$. According to Tables 2 and 3,
when $n$ doubles, the errors of $\alpha_n, \beta_n$ and
$\gamma^{(N)}_n$ decrease by roughly a factor of $2$ while the
errors of the corrected approximations $\tilde{\alpha}_n,
\tilde{\beta}_n$ and $\tilde{\gamma}^{(N)}_n$ decrease by (very)
roughly a factor of $4$. Thus the corrected approximations are much
more accurate than the uncorrected ones. For example,
$\tilde{\alpha}_{100}$ and $\tilde{\beta}_{100}$
($\tilde{\gamma}^{(N)}_{100}$, {\em resp.}) are about as accurate as
or more accurate than $\beta_{400}$ ($\gamma^{(N)}_{400}$, {\em
resp.}).
\end{rem}

\begin{rem}
The discrete approximations are usually computed
using the Markov chain embedding method. A major drawback of this
method is the requirement of a very large state space (corresponding
to a large computer  memory space) for some practical applications.
Indeed, it is shown in \cite{cch2001} that to compute $\alpha_n,
\beta_n$ and $\gamma^{(N)}_n$ using the Markov chain embedding
method, the minimum number of states required is
$\binom{nw}{k-1}+1$, which is enormous when $nw$  is large and $k$
is not small. (It should be remarked that \cite{cch2001} is
concerned with computation of the reliability for the so-called
$d$-within-consecutive-$k$-out-of-$n$ system, which is equivalent to
the discrete scan statistic.) The corrected discrete approximations
partially alleviate the requirement of large memory space since
  a reasonable accuracy can be achieved with relatively small $n$.
\end{rem}

\begin{rem}
Since the assumption of constant intensity plays a relatively minor
role in the proofs of Theorems 2.1, 2.3 and 3.1, we expect that the
method of proof can be extended to the setting of nonhomogeneous
Poisson point processes, which is relevant to computation of the
power of the continuous scan statistic. In the literature, there
appears to be no general method available
 for computing the exact
power under general nonhomogeneous Poisson point processes. The
corrected discrete approximations may prove to be useful in such a
setting as well as in a multiple-window setting ({\em cf.}
\cite{wgf2013}).
\end{rem}

\section*{Acknowledgements}

This work was supported in part by grants from the Ministry of
Science and Technology of Taiwan.

\end{document}